\documentclass[11pt]{amsart}

\usepackage{amsmath}
\usepackage{amsthm}
\usepackage{amssymb}
\usepackage{dsfont}
\usepackage[english]{babel}
\usepackage[utf8x]{inputenc}
\usepackage{overpic}
\usepackage{enumerate}
\usepackage{pdflscape}
\usepackage{float}
\usepackage{hyperref}
\usepackage[all,cmtip]{xy}

\newtheoremstyle{pl}
{3pt}
{3pt}
{\itshape}
{}
{\scshape}
{.}
{.5em}
{}

\newtheoremstyle{pl*}
{3pt}
{3pt}
{\itshape}
{}
{\bfseries}
{.}
{.5em}
{}

\newtheoremstyle{mythm}
{3pt}
{3pt}
{\itshape}
{}
{\bfseries}
{.}
{.5em}
{\thmnote{#1 }#3}

\newtheoremstyle{df}
{3pt}
{3pt}
{\normalfont}
{}
{\scshape}
{.}
{.5em}
{}

\newtheoremstyle{rm}
{3pt}
{3pt}
{\normalfont}
{}
{\scshape}
{.}
{.5em}
{}

\theoremstyle{pl}
\newtheorem{thm}{Theorem}[section]
\newtheorem{lem}[thm]{Lemma}			
\newtheorem{cor}[thm]{Corollary}
\newtheorem{pro}[thm]{Proposition}
\newtheorem{qn}[thm]{Question}
\newtheorem*{qn*}{Question}

\theoremstyle{pl*}
\newtheorem{thm*}{Theorem}
\newtheorem{pro*}[thm*]{Proposition}

\theoremstyle{mythm}
\newtheorem*{mythm}{Theorem}
\newtheorem*{mypro}{Proposition}

\theoremstyle{df}
\newtheorem*{dfn}{Definition}

\theoremstyle{rm}

\newcommand{\ep}{
\epsilon
}
\newcommand{\pa}[1]{
\left(#1\right)
}
\newcommand{\qa}[1]{
\left[#1\right]
}
\newcommand{\ga}[1]{
\left\{#1\right\}
}
\newcommand{\mc}[1]{
\mathcal{#1}
}
\newcommand{\mb}[1]{
\mathbb{#1}
}
\newcommand{\T}{
\mc{T}
}
\newcommand{\mcg}{
\text{\rm Mod}\pa{\Sigma}
}
\newcommand{\psl}{
\text{PSL}_2\pa{\mathbb{C}}
}
\newcommand{\vol}[1]{
\text{\rm vol}\pa{#1}
}

\begin{document}

\title{Volumes of random 3-manifolds}
\author{Gabriele Viaggi}
\date{July 29, 2019}

\begin{abstract}
We prove a law of large numbers for the volumes of families of random hyperbolic mapping tori and Heegaard splittings providing a sharp answer to a conjecture of Dunfield and Thurston.
\end{abstract}

\maketitle

\section{Introduction}
Every orientation preserving diffeomorphism $f\in\text{Diff}^+(\Sigma)$ of a closed orientable surface $\Sigma=\Sigma_g$ of genus $g\ge 2$ can be used to define 3-manifolds in two natural ways: We can construct the {\em mapping torus}
\[
T_f:=\Sigma\times\qa{0,1}/\pa{x,0}\sim\pa{f(x),1},
\]
and we can form the {\em Heegaard splitting}
\[
M_f:=H_g\cup_{f:\partial H_g\rightarrow\partial H_g}H_g.
\]
The latter is obtained by gluing together two copies of the handlebody $H_g$ of genus $g$ along the boundary $\partial H_g=\Sigma$. In both cases the diffeomorphism type of the 3-manifold only depends on the isotopy class of $f$, which means that it is well-defined for the \emph{mapping class} $[f]\in\mcg:=\text{Diff}^+(\Sigma)/\text{Diff}^+_0(\Sigma)$ in the \emph{mapping class group}. We use $X_f$ to denote either $T_f$ or $M_f$.

Invariants of the 3-manifold $X_f$ give rise to well-defined invariants of the mapping class $[f]$. For example, if $X_f$ supports a {\em hyperbolic metric}, then we can use the geometry to define invariants of $[f]$: By Mostow rigidity, if such hyperbolic metric exists, then it is unique up to isometry. 

After Perelman's solution of Thurston's geometrization conjecture, the only obstruction to the existence of a hyperbolic metric on $X_f$ can be phrased in topological terms: A closed orientable 3-manifold is hyperbolic if and only if it is irreducible and atoroidal. Mapping classes that are sufficiently complicated in an appropriate sense (see Thurston \cite{Th} and Hempel \cite{He01}) give rise to manifolds that satisfy these properties. 

For a closed hyperbolic 3-manifold $X_f$, a good measure of its complexity is provided by the {\em volume} $\text{\rm vol}(X_f)$. According to a celebrated theorem by Gromov and Thurston, it equals a universal multiple of the {\em simplicial volume} of $X_f$, a topologically defined invariant (see for example Chapter C of \cite{BP92}). As $X_f$ is not always hyperbolic, in general we define $\text{\rm vol}(X_f)$ to be its simplicial volume, a quantity that always makes sense. 

The purpose of this article is to study the \emph{growth} of the volume for families of {\em random 3-manifolds} or, equivalently, {\em random mapping classes}.

A random mapping class is the result of a {\em random walk} generated by a probability measure on the mapping class group, and a random 3-manifold is one of the form $X_f$ where $f$ is a random mapping class. Such notion of random 3-manifolds has been introduced in the foundational work by Dunfield and Thurston \cite{DT06}. They conjectured that a random 3-manifold is hyperbolic and that its volume grows linearly with the step length of the random walk (Conjecture 2.11 of \cite{DT06}). 

The existence of a hyperbolic metric has been settled by Maher for both mapping tori \cite{Ma} and Heegaard splittings \cite{Ma10}. 

Here we answer to Dunfield and Thurston volume conjecture interpreting it in a strict way (see also Conjecture 9.2 in Rivin \cite{Riv}). Our main result is the following \emph{law of large numbers}: Let $\mu$ be a probability measure on $\mcg$ whose support is a finite symmetric generating set. Let $\omega=(\omega_n)_{n\in\mb{N}}$ be the associated random walk

\begin{thm*}
\label{main1}
There exists $v=v(\mu)>0$ such that for almost every $\omega=\pa{\omega_n}_{n\in\mb{N}}$ the following holds
\[
\lim_{n\rightarrow\infty}\frac{\text{\emph{vol}}\pa{X_{\omega_n}}}{n}=v.
\]
Here $\pa{X_{\omega_n}}_{n\in\mb{N}}$ is either the family of mapping tori or Heegaard splittings.
\end{thm*}


We observe that the asymptotic is the same for both mapping tori and Heegaard splittings. We also remark that the important part is the existence of an {\em exact asymptotic} for the volume as the {\em coarsely linear} behaviour follows from previous work. In the case of mapping tori, it is a consequence of work of Brock \cite{Br}, who proved that there exists a constant $c(g)>0$ such that for every pseudo-Anosov $f$
\[
\frac{1}{c(g)}d_{\text{\rm WP}}(f)\le \vol{T_f}\le c(g)d_{\text{\rm WP}}(f)
\]
where $d_{\text{\rm WP}}(f)$ is the Weil-Petersson translation length of $f$, and the theory of random walks on weakly hyperbolic groups (see for example \cite{MT14}) which provides a linear asymptotic for $d_{\text{\rm WP}}(f)$. 

The coarsely linear behaviour for the volume of a random Heegaard splitting follows from results by Maher \cite{Ma10} combined with an unpublished work of Brock and Souto. We refer to the introduction of \cite{Ma10} for more details. 

Theorem \ref{main1} will be derived from the more technical Theorem \ref{qf tracking} concerning {\em quasi-fuchsian} manifolds. We recall that a quasi-fuchsian manifold is a hyperbolic 3-manifold $Q$ homeomorphic to $\Sigma\times\mb{R}$ that has a {\em compact} subset, the {\em convex core} $\mc{CC}(Q)\subset Q$, that contains all geodesics of $Q$ joining two of its points. The asymptotic geometry of $Q$ is captured by two conformal classes on $\Sigma$, i.e. two points in the Teichmüller space $\T=\T(\Sigma)$. Bers \cite{Bers60} showed that for every ordered pair $X,Y\in\T$ there exists a unique quasi-fuchsian manifold, which we denote by $Q(X,Y)$, realizing those asymptotic data.   

\begin{thm*}
\label{qf tracking}
There exists $v=v(\mu)>0$ such that for every $o\in\T$ and for almost every $\omega=\pa{\omega_n}_{n\in\mb{N}}$ the following limit exists:
\[
\lim_{n\rightarrow\infty}{\frac{\vol{\mc{CC}(Q(o,\omega_no))}}{n}}=v.
\]
\end{thm*}

We remark that $v(\mu)$ is the same as in Theorem \ref{main1}. Once again, the coarsely linear behaviour of the quantity in Theorem \ref{qf tracking} was known before: The technology developed around the solution of the ending lamination conjecture by Minsky \cite{M10} and Brock-Canary-Minsky \cite{BCM}, with fundamental contributions by Masur-Minsky \cite{MM99}, \cite{MM00}, gives a combinatorial description of the internal geometry of the convex core of a quasi-fuchsian manifold. This combinatorial picture is a key ingredient in Brock's proof \cite{Br03} of the following coarse estimate: There exists a constant $k(g)>0$ such that
\[
\frac{1}{k(g)}d_{\text{\rm WP}}(X,Y)-k(g)\le\vol{\mc{CC}(Q(X,Y))}\le k(g)d_{\text{\rm WP}}(X,Y)+k(g).
\]
This link between volumes of hyperbolic 3-manifolds and the Weil-Petersson geometry of Teichmüller space, as in the case of random mapping tori, leads to the coarsely linear behaviour for the volume of the convex cores of $Q(o,\omega_no)$, but does not give, by itself, a law of large numbers. The main novelty in this paper is that we work directly with the geometry of the quasi-fuchsian manifolds rather than their combinatorial counterparts which allows us to get exact asymptotics rather than coarse ones.

The relation between Theorem \ref{main1} and Theorem \ref{qf tracking} is provided by a {\em model manifold} construction similar to Namazi \cite{Na05}, Namazi-Souto \cite{NS09}, Brock-Minsky-Namazi-Souto \cite{BMNS16}. In the case of random 3-manifolds the heuristic picture is the following: The geometry of $X_{\omega_n}$ largely resembles the geometry of the convex core of $Q(o,\omega_no)$, more precisely, as far as the volume is concerned, we have
\[
|\vol{X_{\omega_n}}-\vol{\mc{CC}(Q(o,\omega_no))}|=o(n).
\]

We now describe the basic ideas behind Theorem \ref{qf tracking}: Suppose that the support of $\mu$ equals a finite generating set $S$ and consider $f=s_1\dots s_n$, a long random word in the generators $s_i\in S$. It corresponds to a quasi-fuchsian manifold $Q(o,fo)$. Fix $N$ large, and assume $n=Nm$ for simplicity. We can split $f$ into smaller blocks of size $N$ 
\[
f=(s_1\dots s_N)\cdots(s_{N(m-1)+1}\dots s_{Nm})
\]
which we also denote by $f_j:=s_{jN+1}\cdots s_{(j+1)N}$. Each block corresponds to a quasi-fuchsian manifold $Q(o,f_jo)$ as well. The main idea is that the geometry of the convex core $\mc{CC}(Q(o,fo))$ can be roughly described by juxtaposing, one after the other, the convex cores of the single blocks $\mc{CC}(Q(o,f_jo))$. In particular, the volume $\text{\rm vol}(\mc{CC}(Q(o,fo)))$ can be well approximated by the {\em ergodic sum} 
\[
\sum_{1\le j\le m}{\vol{\mc{CC}(Q(o,f_jo))}}
\]
which converges in average by the Birkhoff ergodic theorem. 

In the paper, we will make this heuristic picture more accurate. Our three main ingredients are the model manifold, bridging between the geometry of the Teichmüller space $\T$ and the internal geometry of quasi-fuchsian manifolds \cite{M10},\cite{BCM}, a recurrence property for random walks \cite{BGH16} and the method of natural maps from Besson-Courtois-Gallot \cite{BCG98}.  They correspond respectively to Proposition \ref{gluing}, Proposition \ref{recurrence} and Proposition \ref{bcg}. Proposition \ref{gluing} and Proposition \ref{recurrence} are used to construct a geometric object, i.e. a negatively curved model for $T_f$, associated to the ergodic sum written above. Proposition \ref{bcg} let us compare this model to the underlying hyperbolic structure.

As an application of the same techniques, along the way, we give another proof of the following well-known result \cite{KoMc}, \cite{BB} relating iterations of pseudo-Anosovs, volumes of quasi-fuchsian manifolds and mapping tori

\begin{pro*}
\label{iteration}
Let $\phi$ be a pseudo-Anosov mapping class. For every $o\in\T$ the following holds:
\[
\lim_{n\rightarrow\infty}{\frac{\vol{\mc{CC}(Q(o,\phi^no))}}{n}}=\vol{T_\phi}.
\]
\end{pro*}

\subsection*{Outline}
The paper is organized as follows. 

In Section \ref{quasi-fuchsian} we introduce quasi-fuchsian manifolds. They are the building blocks for the cut-and-glue construction of Section \ref{gluing and volume}. We prove that, under suitable assumptions, we can glue together a family of quasi-fuchsian manifolds in a geometrically controlled way. The geometric control on the glued manifold is good enough for the application of volume comparison results. 

As an application of the cut-and-glue construction we show that the volume of a random gluing is essentially the volume of a quasi-fuchsian manifold (Proposition \ref{iteration} follows from this fact). As a consequence, in Section \ref{lln volume}, we deduce Theorem \ref{main1} from Theorem \ref{qf tracking} whose proof is carried out shortly after. 

In Section \ref{random walks} we discuss random walks on the mapping class group and on Teichmüller space. The goal is to describe the picture of a random Teichmüller ray and state the main recurrence property.  

In the last section, Section \ref{questions}, we formulate some questions related to the study of {\em growth in random families} of 3-manifolds.

\subsection*{Acknowledgements} I want to thank Giulio Tiozzo for discussing the problem this article is about with me. This work might have never been completed without many discussions with Ursula Hamenstädt. This paper is very much indebted to her. Finally, I thank Joseph Maher for spotting an imprecision in a previous version of the paper and for suggesting to me the article \cite{MS18}. 

\section{Quasi-Fuchisan manifolds}
\label{quasi-fuchsian}

We start by introducing \emph{quasi-fuchsian} manifolds and their geometry. 

\subsection{Marked hyperbolic 3-manifolds}
Let $M$ be a compact, connected, oriented 3-manifold. A marked hyperbolic structure on $M$ is a complete Riemannian metric on $\text{int}(M)$ of constant sectional curvature $\text{sec}\equiv-1$. We regard two Riemannian metrics as equivalent if they are isometric via a diffeomorphism homotopic to the identity.

Every marked hyperbolic structure corresponds to a quotient $\mb{H}^3/\Gamma$ of the hyperbolic 3-space $\mb{H}^3$ by a discrete and torsion free group of isometries $\Gamma<\text{Isom}^+\pa{\mb{H}^3}=\psl$ together with an identification of $\pi_1(M)$ with $\Gamma$, called the {\em holonomy representation} $\rho:\pi_1(M)\longrightarrow\psl$.

We are mostly interested in the cases where $M=\Sigma\times\qa{-1,1}$ is a trivial I-bundle over a surface and when $M$ is closed. By Mostow Rigidity, if $M$ is closed and admits a hyperbolic metric, then the metric is unique up to isometries. In this case we denote by $\vol{M}$ the volume of such a metric.

\subsection{Quasi-fuchsian manifolds}
A particularly flexible class of structures is provided by the so-called \emph{quasi-fuchsian} manifolds 

\begin{dfn}[Quasi-Fuchsian]
A marked hyperbolic structure $Q$ on $\Sigma\times\qa{-1,1}$ is \emph{quasi-fuchsian} if $\mb{H}^3/\rho(\pi_1(\Sigma))$ contains a {\em compact} subset which is {\em convex}, that is, containing every geodesic joining a pair of points in it. The smallest convex subset is called the \emph{convex core} and is denoted by $\mc{CC}(Q)$.

The convex core $\mc{CC}(Q)$ is always a topological submanifold. If it has codimension $1$ then it is a totally geodesic surface and we are in the \emph{fuchsian} case, the group $\Gamma<\text{\rm Isom}^+\pa{\mb{H}^3}$ stabilizes a totally geodesic $\mb{H}^2\subset\mb{H}^3$. In the generic case it has codimension $0$ and is homeomorphic to $\Sigma\times\qa{-1,1}$. The inclusion $\mc{CC}(Q)\subset Q$ is always a homotopy equivalence.
\end{dfn}

We denote by 
\[
\vol{Q}:=\vol{\mc{CC}(Q)}\in[0,\infty)
\]
the volume of the convex core of the quasi-fuchsian manifold $Q$.  

\subsection{Deformation space}
We denote by $\T$ the Teichmüller space of $\Sigma$, that is, the space of marked hyperbolic structures on $\Sigma$ up to isometries homotopic to the identity. We equip $\T$ with the Teichmüller metric $d_\T$.

To every quasi-fuchsian manifold $Q$ one can associate the \emph{conformal boundary} $\partial_cQ$ in the following way: The surface group $\pi_1(\Sigma)$ acts on $\mb{H}^3$ by isometries and on $\mb{CP}^1=\partial\mb{H}^3$ by Möbius transformations. It also preserves a convex set, the lift of $\mc{CC}(Q)$ to the universal cover, on which it acts cocompactly. By Milnor-\v{S}varc, for any fixed basepoint $o\in\mb{H}^3$, the orbit map $\gamma\in\pi_1(\Sigma)\rightarrow\gamma o\in\mb{H}^3$ is a quasi-isometric embedding and extends to a topological embedding on the boundary $\partial\pi_1(\Sigma)\hookrightarrow\mb{CP}^1$. The image is a topological circle $\Lambda$, called the \emph{limit set}, that divides the Riemann sphere $\mb{CP}^1$ into a union of two topological disks $\Omega=\mb{CP}^1\setminus\Lambda$. The action $\pi_1(\Sigma)\curvearrowright\Omega$ preserves the connected components, and is free, properly discontinuous and conformal. The quotient $\partial_cQ=\Omega/\pi_1(\Sigma)=X\sqcup Y$ is a disjoint union of two marked oriented Riemann surfaces, homeomorphic to $\Sigma$, and it is called the \emph{conformal boundary} of $Q$. The quotient $\bar{Q}:=(\mb{H}^3\cup\Omega)/\Gamma$ compactifies $Q$.

\begin{thm}[Double Uniformization, Bers \cite{Bers60}]
\label{bers}
For every ordered pair of marked hyperbolic surfaces $(X,Y)\in\T\times\T$ there exists a unique equivalence class of quasi-fuchsian manifolds, denoted by $Q(X,Y)$, realizing the conformal boundary $\partial_cQ(X,Y)=X\sqcup Y$.
\end{thm} 

The mapping class group $\mcg$ acts on quasi-fuchsian manifolds by precomposition with the marking. In Bers coordinates it plainly translates into $\phi Q(X,Y)=Q\pa{\phi X,\phi Y}$.

\subsection{Teichmüller geometry and volumes}
Later, it will be very important for us to quantify the price we have to pay in terms of volume if we want to replace a quasi-fuchsian manifold $Q$ with another one $Q'$. We would like to express $\left|\vol{Q}-\vol{Q'}\right|$ in terms of the geometry of the conformal boundary.

Despite the fact that Weil-Petersson geometry is more natural when considering questions about volumes, we will mainly use the Teichmüller metric $d_\T$. The reason is that we are mostly concerned with {\em upper bounds} for the volumes of the convex cores. It is a classical result of Linch \cite{L74} that the Teichm\"uller distance is bigger than the Weil-Petersson distance $d_{\text{\rm WP}}\le\sqrt{2\pi|\chi(\Sigma)|}d_\T$. The following is our main tool:

\begin{pro}[Proposition 2.7 in Kojima-McShane \cite{KoMc}, see also Schlenker \cite{S13}]
\label{replacement}
There exists $\kappa=\kappa(\Sigma)>0$ such that
\[
\left|\vol{Q(X,Y)}-\vol{Q(X',Y')}\right|\le \kappa\pa{d_\T(X,X')+d_\T(Y,Y')}+\kappa.
\]
\end{pro}

This formulation is not literally Proposition 2.7 of \cite{KoMc} so we spend a couple of words to explain the two diffenrences. Firstly, the estimate in Proposition 2.7 of \cite{KoMc} concerns the {\em renormalized volume} and not volume of the convex core. However, the two quantities only differ by a uniform additive constant (see Theorem 1.1 in \cite{S13}). Secondly, their statement is limited to the case where $X=X'=Y'$, but their proof exteds word by word to the more general setting: It suffices to apply their argument to the one parameter families $Q(X,Y_t)$ and $Q(X_t,Y')$, where $X_t$ and $Y_t$ are the Teichm\"uller geodesics joining $X$ to $X'$ and $Y$ to $Y'$. 

\subsection{Geometry of the convex core}
We associate to the quasi-fuchsian manifold $Q=Q(X,Y)$ the Teichmüller geodesic $l:[0,d]\rightarrow\T$ joining $X$ to $Y$ where $d=d_\T(X,Y)$. Work of Minsky \cite{M10} and Brock-Canary-Minsky \cite{BCM} relates the geometry of the Teichm\"uller geodesic $l$ to the internal geometry of $\mc{CC}(Q)$. In the next section we will use this information to glue together convex cores of quasi-fuchsian manifolds in a controlled way. 

As a preparation, we start with a description of the boundary $\partial\mc{CC}(Q)$ and introduce some useful notation. We recall that, topologically, $\mc{CC}(Q)\simeq\Sigma\times[-1,1]$. The convex core separates $\bar{Q}=Q\cup\partial_cQ$ into two connected components, containing, respectively, $X$ and $Y$. We denote by $\partial_X\mc{CC}(Q)$ and $\partial_Y\mc{CC}(Q)$ the components of $\partial\mc{CC}(Q)$ that face, respectively, $X$ and $Y$. As observed by Thurston, the surfaces $\partial_X\mc{CC}(Q)$ and $\partial_Y\mc{CC}(Q)$, equipped with the induced path metric, are hyperbolic. By a result of Sullivan, they are also uniformly bilipschitz equivalent $X$ and $Y$ (see Chapter II.2 of \cite{CEG}).

\section{Gluing and Volume}
\label{gluing and volume}

This section describes a gluing construction (Proposition \ref{gluing}) which is a major technical tool in the article. It allows us to cut and glue together quasi-fuchsian manifolds in a sufficiently controlled way. The control on the {\em models} obtained with this procedure is then exploited to get volume comparisons via the method of natural maps (Proposition \ref{bcg} as in \cite{BCG98}) which is the second major tool of the section. 

Along the way we recover a well-known result (Proposition \ref{iteration}) relating iterations of pseudo-Anosov maps and volumes of quasi-fuchsian manifolds.      

\subsection{Product regions and Cut and Glue construction}

The cut and glue construction we are going to describe is a standard way to glue Riemannian 3-manifolds. Here we import the discussion and some of the observations of Section 5 of \cite{HV} and adapt them to our special setting. We start with a pair of definitions.

\begin{dfn}[Product Region]
Let $Q$ be a quasi-fuchsian manifold. A \emph{product region} $U\subset Q$ is a codimension 0 submanifold homeomorphic to $\Sigma\times\qa{0,1}$ whose inclusion in $Q$ is a homotopy equivalence.

Using the orientation and product structure of $Q$ we can define a \emph{top boundary} $\partial_+U$ and a \emph{bottom boundary} $\partial_-U$. We denote by $Q_-$ and $Q_+$ the parts of $Q$ that lie \emph{below} $\partial_+U$ and \emph{above} $\partial_-U$ respectively.

A product region comes together with a {\em marking}, an identification $j_U:\pi_1(\Sigma)\stackrel{\sim}{\rightarrow}\pi_1(U)$, defined as follows: The data of a marked hyperbolic structure $Q$ gives us an identification $\pi_1(\Sigma)\simeq\pi_1(Q)$ and the inclusion $U\subset Q$, being a homotopy equivalence, gives $\pi_1(Q)\simeq\pi_1(U)$. The marking allows us to talk about the homotopy class of a map between product regions. 
\end{dfn}

Any homotopy equivalence $k:U\rightarrow V$ induces a well-defined mapping class $\qa{k}\in\mcg\simeq\text{\rm Out}^+(\pi_1(\Sigma))$ (Dehn-Nielsen-Baer, Theorem 8.1 in \cite{primer}), namely, the one corresponding to the outer automorphism 
\[
\pi_1(\Sigma)\stackrel{j_U}{\simeq}\pi_1(U)\stackrel{k}{\simeq}\pi_1(V)\stackrel{j_V}{\simeq}\pi_1(\Sigma).
\]

We also want to quantify the geometric quality of a map between product regions. Since we want to keep the curvature tensor under control, a good measurement for us is provided by the $\mc{C}^2$-norm. 

\begin{dfn}[Almost-Isometric]
Let $k:(U,\rho_U)\rightarrow(V,\rho_V)$ be a smooth embedding between Riemannian manifolds. Denote by $\nabla_U,\nabla_V$ the Levi-Civita connections. Consider the $\mc{C}^2$-norm
\[
\left|\left|\rho_U-k^*\rho_V\right|\right|_{\mc{C}^2}:=\left|\left|\rho_U-k^*\rho_V\right|\right|_{\mc{C}^0}+\left|\left|\nabla_Uk^*\rho_V\right|\right|_{\mc{C}^0}+\left|\left|\nabla_U\nabla_Uk^*\rho_V\right|\right|_{\mc{C}^0}.
\]
For $\xi>0$ we say that $k$ is $\xi$-almost isometric if $\left|\left|\rho_U-k^*\rho_V\right|\right|_{\mc{C}^2}<\xi$.
\end{dfn}

The following lemma is what we refer to as the cut-and-glue construction.

\begin{lem}
\label{cut and glue}
Let $Q,Q'$ be quasi-fuchsian manifolds. Denote by $\rho_Q,\rho_{Q'}$ their Riemannian metrics. Suppose we have product regions $U\subset Q$, $U'\subset Q'$ and a diffeomorphism $k:U\rightarrow U'$ between them. Suppose also that $\theta:U\rightarrow\qa{0,1}$ is a smooth function with $\theta|_{\partial_-U,\partial_+U}\equiv 0,1$. Then we can form the 3-manifold 
\[
Q''=Q_-\cup_{k:U\rightarrow U'}Q'_+
\]
and endow it with the Riemannian metric
\[
\rho:=
\left\{
\begin{array}{l l}
\rho_Q &\text{ on }Q_-\setminus U\\
(1-\theta)\rho_Q+\theta k^*\rho_{Q'} &\text{ on }U\\
\rho_{Q'} &\text{ on }Q'_+\setminus U'.\\
\end{array}
\right.
\]
If $k$ is $\xi$-almost isometric for some $\xi<1$, then, on $U\subset Q''$, we have the following sectional curvature and diameter bounds
\[
\left|1+\text{\emph{sec}}_{Q''}\right|\le c_3\left|\left|\theta\right|\right|_{\mc{C}^2}\cdot\left|\left|\rho_Q-k^*\rho_{Q'}\right|\right|_{\mc{C}^2}
\]
for some universal constant $c_3$ and
\[
\text{\rm diam}_\rho(U)\le(1+\xi)\text{\rm diam}_{\rho_U}(U).
\]
In particular, if $\text{\rm diam}_{\rho_U}(U)$ is uniformly bounded, the same is true for $\text{\rm vol}_\rho(U)$.
\end{lem}

We associate two parameters to a product region, {\em diameter} and {\em width}
\begin{align*}
&\text{diam}(U):=\sup\ga{d_Q(x,y)\left|x,y\in U\right.},\\
&\text{width}(U):=\inf\ga{d_Q(x,y)\left|x\in\partial_+U,y\in\partial_-U\right.}.
\end{align*}
If a product region has width at least $D$ and diameter at most $2D$ we say that it has {\em size} $D$. The Margulis Lemma implies that the injectivity radius of a product region of size $D$, defined as
\[
\text{inj}(U):=\inf_{x\in U}\ga{\text{inj}_x(Q)},
\]
is bounded from below in terms of $D$

\begin{lem}
\label{injrad}
For every $D>0$ there exists $\ep_0(D,g)>0$ such that a product region $U$ of size $D$ has $\text{\rm inj}(U)>\ep_0$.  
\end{lem}

\begin{proof}
The inclusion of $U$ in $Q$ is $\pi_1$-surjective. Having diameter bounded by $2D$, the region $U$ cannot intersect too deeply any very thin Margulis tube $\mb{T}_\gamma$ otherwise $\pi_1(U)\rightarrow\pi_1(Q)$ would factor through $\pi_1(U)\rightarrow\pi_1(\mb{T}_\gamma)$.
\end{proof}

In particular, a compactness argument with the geometric topology on pointed hyperbolic manifolds gives us the following property: Once we fix the size of a product region we can produce a uniform bump function on it.

\begin{lem}[Lemma 5.2 of \cite{HV}] 
\label{bump function}
For every $D>0$ there exists $K>0$ such that the following holds: Let $U\simeq\Sigma\times\qa{0,1}$ be a product region of size $D$. Then there exists a smooth function $\theta:U\rightarrow\qa{0,1}$ with the following properties:
\begin{itemize}
\item{Near the boundaries it is constant: $\theta|_{\partial_-U}\equiv 0$ and $\theta|_{\partial_+U}\equiv 1$.}
\item{Uniformly bounded $\mc{C}^2$-norm $\left|\left|\theta\right|\right|_{\mc{C}^2}\le K$.}
\end{itemize}
\end{lem}

\subsection{Almost-isometric embeddings of product regions}
For a fixed $\eta>0$ we denote by $\T_\eta$ the {\em $\eta$-thick part} of Teichmüller space consisting of those hyperbolic structures with no geodesic shorter than $\eta$. 

The following is a consequence of the {\em model manifold} technology developed by Minsky \cite{M10} around the solution of the Ending Lamination Conjecture (completed then in Brock-Canary-Minsky \cite{BCM}).
  
\begin{pro}[see Proposition 6.2 \cite{HV}]
\label{minsky model}
For every $\eta,\xi,\delta,D>0$ there exists $D_0(\eta,g)$ and $h=h(\eta,\xi,\delta,D)>0$ such that the following holds: Let $Q_1$, $Q_2$ be quasi-fuchsian manifolds with associated Teichm\"uller geodesics $l_i:I_i\subseteq\mb{R}\rightarrow\T$ with $i=1,2$. Suppose that $l_1,l_2$ $\delta$-fellow travel on a subsegment $J$ of length at least $h$ and entirely contained in $\T_\eta$. Then there exist product regions $U_i\subset\mc{CC}(Q_i)$ of size $D$ and a $\xi$-almost isometric embedding $k:U_1\rightarrow U_2$ in the homotopy class of the identity. Moreover, if $D\ge D_0$ we can assume that $U_i$ contains the geodesic representative of $\alpha$, a curve which has moderate length for both $Q_i$ and $T\in J$ the midpoint of the segment, i.e. $l_{Q_i}(\alpha)$, $L_T(\alpha)\le D_0$.
\end{pro}

In the statement and in the next section we use the following notation:

\begin{center}
\label{notation1}
\begin{minipage}{.8\linewidth}
\textbf{Notation}. If $\alpha:S^1\rightarrow Q$ is a closed loop in a hyperbolic 3-manifold, we denote by $l(\alpha)$ its length and by $l_Q(\alpha)$ the length of the unique geodesic representative in the homotopy class. If the target instead is a hyperbolic surface $\alpha:S^1\rightarrow X$, we use the notations $L(\alpha)$ and $L_X(\alpha)$.
\end{minipage}
\end{center}

For a proof we refer to \cite{HV}. The geodesic $\alpha$ is used to locate the product regions inside the convex cores. We explain that in the following section. For now we remark the following immediate consequence:

\begin{dfn}[$\eta$-Height]
Let $l:I\rightarrow\T$ be a Teichmüller geodesic. The {\em $\eta$-height} of $l$ is the length of the maximal connected subsegment of $I$ whose image is entirely contained in $\T_\eta$.
\end{dfn}

\begin{cor}
\label{wide cores}
Fix $\eta>0$. There exists a function $\rho:(0,\infty)\rightarrow(0,\infty)$ with $\rho(h)\uparrow\infty$ as $h\uparrow\infty$ and the following property: Let $Q=Q(X,Y)$ be a quasi-fuchsian manifold with associated geodesic $l:I\rightarrow\T$. Suppose that the $\eta$-height is at least $h$ then 
\[
d_Q\pa{\partial_X\mc{CC}(Q),\partial_Y\mc{CC}(Q)}\ge\rho(h).
\]
\end{cor}

\subsection{Position of the product region}
From now on we fix once and for all a sufficiently large size $D_1\ge D_0$ for the product regions we consider. 

Let $\alpha:S^1\rightarrow Q$ be a non-trivial closed curve in a hyperbolic $3$-manifold $Q$ that has a geodesic representative $\alpha^*\subset Q$. By basic hyperbolic geometry
\[
\cosh\pa{d_Q(\alpha,\alpha^*)}l_Q(\alpha)\le l(\alpha).
\]
Suppose that $Q=Q(X,Y)$ is a quasi-fuchsian manifold. Let $U\subset Q$ be a product region of size $D_1$ containing a closed geodesic $\alpha$. By the assumption on the size of $U$ and Lemma \ref{injrad} we have $l_Q(\alpha)\ge 2\ep_0(D_1,g)$. Recall that $\partial_X\mc{CC}(Q)$ denotes the boundary of the convex core that faces the conformal boundary $X$. By a Theorem due to Sullivan (see Chapter II.2 and in particular Theorem II.2.3.1 in \cite{CEG}), there exists a universal constant $K$ such that $\partial_X\mc{CC}(Q)$ and $X$ are $K$-bilipschitz equivalent via a homeomorphism in the homotopy class of the identity. We have
\[
d_Q(\partial_X\mc{CC}(Q),\alpha)\le\text{\rm arccosh}\pa{\frac{L_{\partial_X\mc{CC}(Q)}(\alpha)}{l_Q(\alpha)}}\le\text{\rm arccosh}\pa{\frac{KL_X(\alpha)}{2\ep_0(D_1,g)}}.
\]
Let $T\in\T$ be a hyperbolic structure for which $L_T(\alpha)\le D_0(\eta,g)$. Wolpert's inequality $L_X(\alpha)\le L_T(\alpha)e^{2d_\T(X,T)}$ (see Lemma 12.5 in \cite{primer}) allows us to continue the chain of inequalities to the following:
\[
d_Q(\partial_X\mc{CC}(Q),\alpha)\le\text{\rm arccosh}\pa{\frac{KD_0(\eta,g)}{2\ep_0(D_1,g)}e^{2d_\T(X,T)}}.
\]
Let us introduce the function $F:(0,\infty)\rightarrow(0,\infty)$ defined by 
\[
F(t)=\text{\rm arccosh}\pa{\frac{KD_0(\eta,g)}{2\ep_0(D_1,g)}e^{2t}}.
\]
With this notation we have

\begin{lem}
\label{distance boundary}
Let $U\subset Q(X,Y)$ be a product region of size $D_1$ containing a closed geodesic $\alpha\subset U$. Let $T\in\T$ be a surface such that $L_T(\alpha)\le D_0$. Then 
\[
d_Q(\partial_X\mc{CC}(Q),U)\le F(d_\T(X,T)).
\]
\end{lem} 

Combining Corollary \ref{wide cores} and Lemma \ref{distance boundary} we can ensure that a pair of product regions is well separated. To this extent we introduce the function $G:(0,\infty)\rightarrow(0,\infty)$ defined by 
\[
G(t)=\inf_{t\in\mb{R}}\ga{\text{for every $s>t$ we have $\rho(s)>2F(t)+4D_1$}}.
\]

\begin{lem}
\label{disjoint}
Let $U^-$, $U^+$ be product regions of size $D_1$ in $Q=Q(X^-,X^+)$. Suppose they contain, respectively, closed geodesics $\alpha^-$, $\alpha^+$. Let $T^-$, $T^+\in\T$ be surfaces such that $L_{T^-}(\alpha^-)$, $L_{T^+}(\alpha^+)\le D_0$. Consider 
\[
d:=\max\ga{d_\T(X^-,T^-),d_\T(X^+,T^+)}.
\]
If the $\eta$-height $h$ of $Q$ is at least $h\ge G(d)$ then the product regions are disjoint and cobound a codimension $0$ submanifold $Q^0\subset Q$ homeomorphic to $\Sigma\times\qa{0,1}$ for which $U^-$, $U^+$ are collars of the boundary. 
\end{lem}

\begin{proof}
We have $d_Q(\partial_{X^-}\mc{CC}(Q),\partial_{X^+}\mc{CC}(Q))\ge\rho(h)$ and $d_Q(\partial_{X^\pm}\mc{CC}(Q),U^\pm)\le F(d)$. If $\rho(h)-F(d)-2D_1\ge F(d)+2D_1$, the product regions $U^-,U^+$ are separated. By definition of $G$, if $h>G(d)$, the previous inequality holds.
\end{proof}

Finally we take care of the volume. 

\begin{lem}
\label{volume difference}
If $X^-,X^+\in\T_\eta$, then there exists $V_0(D_1,\eta,d)$ such that
\[
\left|\vol{Q}-\vol{Q^0}\right|\le V_0.
\]
\end{lem}

\begin{proof}
There is a uniform upper bound on the diameter of a $\eta$-thick hyperbolic surface. By Sullivan, the same holds for every component of $\partial\mc{CC}(Q)$. It follows that the diameter of the region enclosed by $U^-$ and $\partial_{X^-}\mc{CC}(Q)$ is uniformly bounded in terms of $D_1,\eta$. As an upper bound for its volume we can take the volume of a ball with the same radius in $\mb{H}^3$. 
\end{proof}

\subsection{A gluing theorem}
Recall that we fixed $D_1>0$ sufficiently large once and for all. The following is our first crucial technical tool.

\begin{figure}[H]
\begin{center}
\begin{overpic}{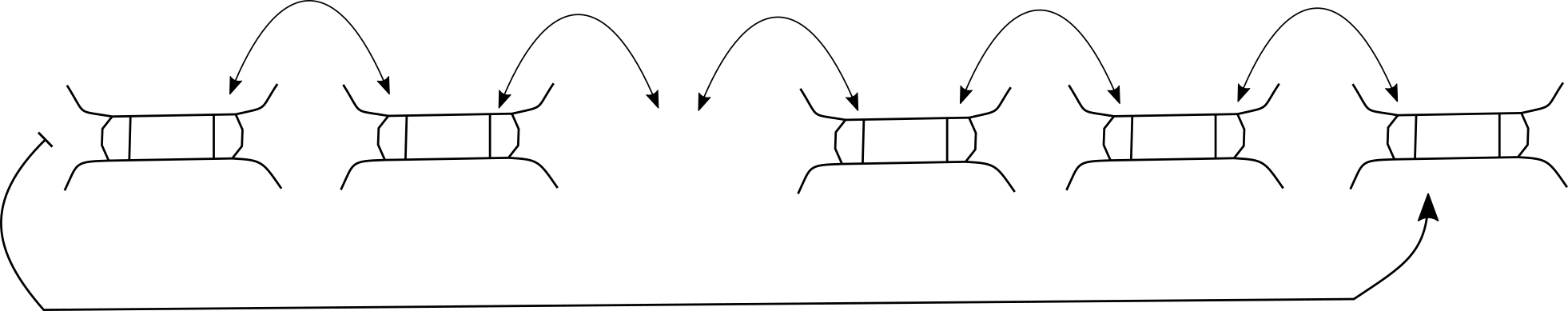}
\put (41.5,10) {$\dots$}
\put (9,6) {$Q_1$}
\put (27,6) {$Q_2$}
\put (56,6) {$Q_{r-1}$}
\put (73.5,6) {$Q_r$}
\put (93,6) {$\phi Q_1$}
\put (42,2) {$\phi$}
\end{overpic}
\label{fig:figure0}
\end{center}
\end{figure}

\begin{figure}[H]
\begin{center}
\begin{overpic}{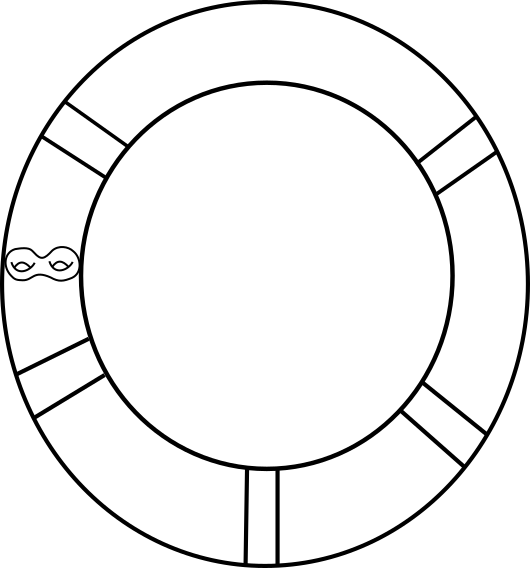}
\put (85,41) {$\vdots$}
\put (-33,12) {$X_\phi\simeq T_\phi$}
\end{overpic}
\caption{Gluing.}
\label{fig:figure1}
\end{center}
\end{figure}

\begin{pro}
\label{gluing}
Fix $\eta,\delta>0$ and $\xi\in(0,1)$. There exists $h_0(\eta,\xi,\delta)>0$ such that the following holds: Let $\ga{Q_i=Q(X_i^-,X_i^+)}_{i=1}^r$ be a family of quasi-fuchsian manifolds. Let $\ga{l_i:I_i\rightarrow\T}_{i=1}^r$ be the corresponding Teichm\"uller geodesics. Suppose that the following holds:
\begin{itemize}
\item{For every $i<r$, the geodesics $l_i,l_{i+1}$ $\delta$-fellow travel when restricted to $J_i^+\subset I_i$ and $J_{i+1}^-\subset I_{i+1}$. The segments $J_i^+$ and $J_{i+1}^-$ are respectively terminal and initial, have length $\left|J_i^-\right|$, $\left|J_i^+\right|\in\qa{h_0,2h_0}$ and are entirely contained in $\T_\eta$.}
\item{The $\eta$-height of $l_i$ is at least $G(2h_0)$ for all $i\le r$.}
\end{itemize}
Let $k_i:U_i^+\subset Q_i\rightarrow U_{i+1}^-\subset Q_{i+1}$ be the $\xi$-almost isometric embedding of product regions in the homotopy class of the identity for $i<r$  corresponding to the segments $J_i^+$, $J_{i+1}^-$ as in Proposition \ref{minsky model}. The product regions have size $D_1$ and are disjoint as in Lemma \ref{disjoint}. Let $Q_i^0$ be the region of $Q_i$ bounded by $\partial^-U_i^-$ and $\partial^+U_i^+$ for which $U_i^-$, $U_i^+$ are collars of the boundaries as in Lemma \ref{disjoint}. Then we can form
\[
X:=Q_1^0\cup_{k_1:U_1^+\rightarrow U_2^-}Q_2^0\cup\dots\cup Q_{r-1}^0\cup_{k_{r-1}:U_{r-1}^+\rightarrow U_r}Q_r^0
\]
using the cut and glue construction Lemma \ref{cut and glue}. The compact 3-manifold $X$ has the following properties:
\begin{itemize}
\item{Curvature: $\left|1+\text{\rm sec}_X\right|\le K\xi$ where $K=K(D_1)$ is as in Lemma \ref{bump function}.}
\item{The inclusions $Q_i^0\setminus\pa{U_i^-\cup U_i^+}\subset X$ are isometric.}
\item{Volume: There exists $V_0=V_0(\eta,\xi,D_1,h_0)$ such that
\[
\left|\vol{X}-\sum_{i<r}{\vol{Q_i}}\right|\le rV_0.
\]
}
\end{itemize}
Furthermore, let $\phi\in\mcg$ be a mapping class. Suppose it has the property that $\phi l_1$ and $l_r$ $\delta$-fellow travel along $J_1^-\subset I_1$ and $J_r^+\subset I_r$. Then there is a $\xi$-almost isometric embedding $k_r:U_r^+\subset Q_r\rightarrow U_1^-\subset Q_1$ in the homotopy class of $\phi$ and we can form the manifold
\[
X_\phi=X/\pa{k_r:U_r^+\subset Q_r^0\rightarrow U_1^-\subset Q_1^0}.
\]
Topologically $X_\phi$ is diffeomorphic to the mapping torus of $\phi$.
\end{pro}

The $\xi$-almost isometric embedding $k_r$ is obtained as the composition of the one provided by Proposition \ref{minsky model} for the fellow traveling of $l_r$, $\phi l_1$ and the isometric remarking $\phi Q_1\rightarrow Q_1$ in the isotopy class of $\phi$ (see Figure \ref{fig:figure1}). 

Proposition \ref{gluing} follows directly from several applications of Proposition \ref{minsky model} and the cut and glue construction Lemma \ref{cut and glue} once we can ensure that the product regions are well separated as in Lemma \ref{disjoint}. Separation and volume bounds follow from the discussion in the previous section.

We remark that, by a celebrated theorem of Thurston \cite{Th}, if $\phi$ is a {\em pseudo-Anosov} mapping class, then the mapping torus $T_\phi$ admits a hyperbolic metric. A pseudo-Anosov element $\phi$ is one that acts as a hyperbolic isometry of Teichmüller space: It preserves a unique Teichmüller geodesic $l:\mb{R}\rightarrow\T$ on which it acts by translations $\phi l(t)=l(t+L(\phi))$. The quantity $L(\phi)>0$ is called the {\em translation length} of $\phi$ (see Chapter 13 of \cite{primer}).

\subsection{Comparing the volume}
The second fundamental ingredient is a volume comparison result. If we have two Riemannian metrics $g_0$ and $g$ on the same 3-manifold $M$ we can compare their volume using the method of {\em natural maps} introduced by Besson, Courtois and Gallot. We mainly refer to their work \cite{BCG98} as we use some consequences of it. Given a map $f:N\rightarrow M$ between Riemannian manifolds satisfying certain curvature conditions, the method produces families of natural maps $F:N\rightarrow M$ homotopic to $f$ and with Jacobian bounded in terms of the {\em volume entropies} of the manifolds. We need the following result:  

\begin{thm}[Besson-Courtois-Gallot \cite{BCG98}]
\label{bcg}
Let $(M,g)$ and $(M_0,g_0)$ be closed orientable Riemannian 3-manifolds such that there exists:
\begin{itemize}
\item{A lower bound for the Ricci curvature of the source $\text{\rm Ric}_g\ge -2g$.}
\item{A uniform bound for the sectional curvatures of the target $-k\le\text{\rm sec}_{g_0}\le-1$ for some $k\ge 1$.}
\end{itemize}
Then for every continuous map $f:M\longrightarrow M_0$ we have
\[
\vol{M}\ge\left|\text{\rm deg}(f)\right|\vol{M_0}.
\]
\end{thm}

We now describe some applications. 

The first one is to the models constructed in Proposition \ref{gluing}:

\begin{cor}
\label{special case}
If $\phi$ is a pseudo-Anosov mapping class, $X_\phi$ is as in Proposition \ref{gluing} and $K\xi<1$ then 
\[
(1-K\xi)^{-3/2}\text{\rm vol}\pa{X_\phi}\le\text{\rm vol}\pa{M_\phi}\le(1+K\xi)^{3/2}\text{\rm vol}\pa{X_\phi}.
\] 
\end{cor}

\begin{proof}
The mapping torus of $\phi$ admits a purely hyperbolic Riemannian metric and the metric $X_\phi$ with $\text{\rm sec}_{X_\phi}\in\pa{-1-K\xi,-1+K\xi}$. We apply Theorem \ref{bcg} to the identity map in both directions after suitably rescaling the metric on $X_\phi$ so that it fulfills the Ricci and sectional curvature bounds.  
\end{proof}

The second application is a construction of a very peculiar model of a mapping torus $T_\phi$ of a pseudo-Anosov diffeomorphism $\phi$. Recall that $\phi$ acts on its axis by translating points by $L(\phi)$.

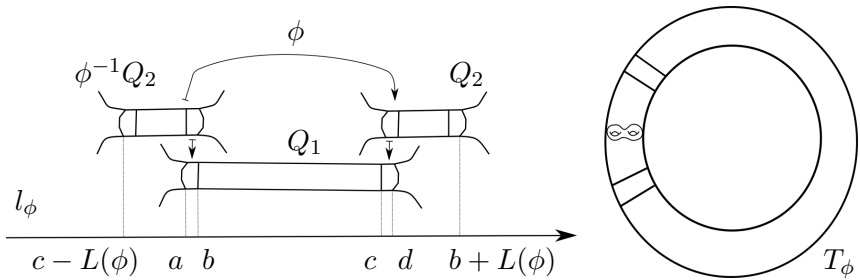
\begin{figure}[H]
\begin{center}
\begin{overpic}{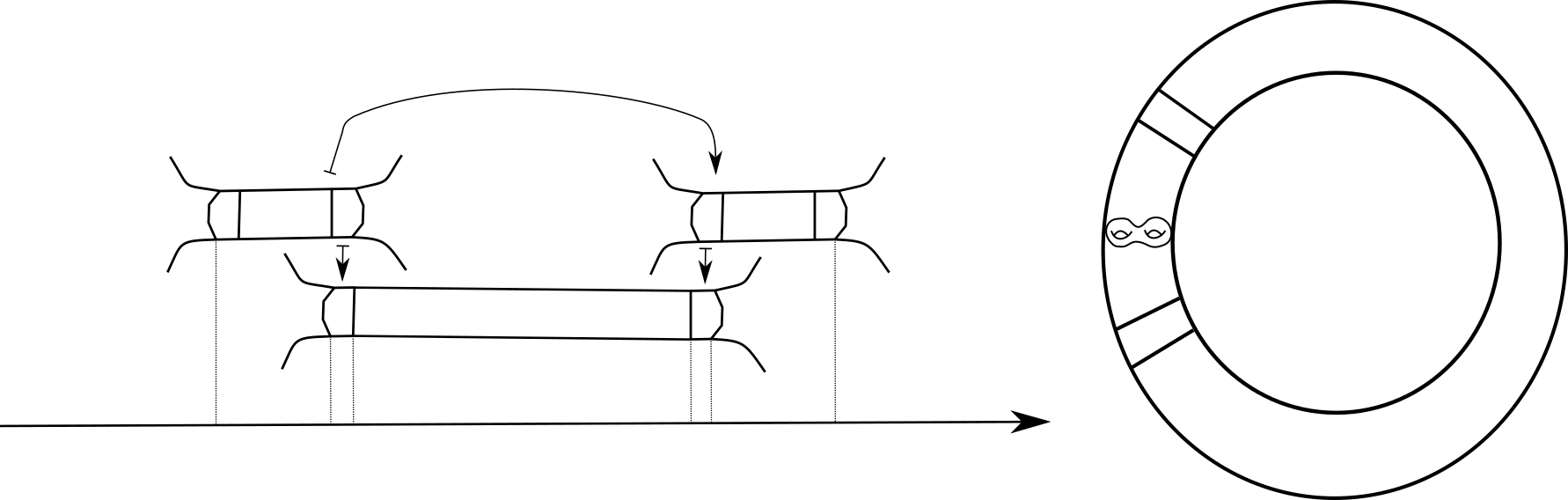}
\put (1,8) {$l_\phi$}
\put (3,1) {$c-L(\phi)$}
\put (19,1) {$a$}
\put (8,23) {$\phi^{-1}Q_2$}
\put (52,23) {$Q_2$}
\put (33,28) {$\phi$}
\put (33,15) {$Q_1$}
\put (23,1) {$b$}
\put (42,1) {$c$}
\put (46,1) {$d$}
\put (52,1) {$b+L(\phi)$}
\put (96,1) {$T_\phi$}
\end{overpic}
\end{center}
\caption{Model for a mapping torus.}
\label{fig:figure7}
\end{figure}

\begin{cor}
\label{qf vs mt}
Fix $\eta>0$ and $\xi\in(0,1)$. There exists $h(\xi,\eta)>0$ such that the following holds: Let $\phi$ be a pseudo-Anosov with axis $l:\mb{R}\rightarrow\T$. Suppose that there are disjoint intervals $I=\qa{a,b}$ and $J=\qa{c,d}$ with $a<b<c<d<a+L(\phi)$ such that $l(I),l(J)\subset\T_\eta$ and $\left|I\right|$, $\left|J\right|\ge h$. Then
\[
\left|\vol{T_\phi}-\vol{Q(l(a),l(d))}\right|\le\kappa(L(\phi)+b-c)+\xi\kappa(d-a)+\text{\rm const}
\]
where $\text{\rm const}$ depends only on $\eta,\xi,h,D_1$.
\end{cor}

\begin{proof}
Let $h_0(\eta,\xi,0)$ be as in Proposition \ref{gluing}. If $h\ge\max\{h_0,G(2h_0)\}$ is large enough, then the quasi-fuchsian manifolds (see Figure \ref{fig:figure7})
\[
\ga{Q_1=Q(l(a),l(d)),Q_2=Q(l(c),l(L(\phi)+b))}
\]
satisfy the assumption of Proposition \ref{gluing}. Moreover $\phi Q_1=Q(l(a+L(\phi),d+L(\phi))$ and the segments $\qa{l(c),l(b+L(\phi))}$ and $\qa{l(a+L(\phi)),l(d+L(\phi))}$ overlap along $\qa{l(a+L(\phi)),l(b+L(\phi))}=\phi\qa{l(a),l(b)}$. The upper bound for the volume is just an application of Proposition \ref{replacement}
\begin{flalign*}
&\left|\vol{T_\phi}-\vol{Q(l(a),l(d))}\right|\\
&\le\vol{Q(l(c),l(L(\phi)+b))}+2V_0+\xi\vol{Q(l(a),l(d))}\\
&\le\kappa(L(\phi)+b-c)+\xi\kappa(d-a)+2V_0+2\kappa.
\end{flalign*} 
\end{proof}

Using this estimates we recover the following well-known result (see for example \cite{BB}, \cite{KoMc}):

\begin{mypro}[3]
Let $\phi$ be a pseudo-Anosov mapping class. Then for every $o\in\T$ we have 
\[
\lim{\frac{\vol{Q(o,\phi^no)}}{n}}=\vol{T_\phi}.
\]
\end{mypro}

\begin{proof}
There exists $\eta_\phi>0$ such that $l_\phi:\mb{R}\rightarrow\T$, the Teichmüller axis of $\phi$, lies in $\T_{\eta_\phi}$. Fix $\xi>0$ and consider $h=h(\eta_\phi,\xi)$. For $n$ large enough the intervals $I=\qa{0,h}$ and $J=\qa{nL(\phi)-h,nL(\phi)}$ fulfill the assumption of Corollary \ref{qf vs mt} with respect to $\phi^n$. Hence, for all large $n$, $\left|\vol{Q(l_\phi(0),l_\phi(nL(\phi)))}-n\vol{T_\phi}\right|\le\kappa 2h+\xi\kappa nL(\phi)+\text{\rm const}$.
Observe that $l_\phi(nL(\phi))=\phi^n l_\phi(0)$. Denote $l_\phi(0)$ by $o_1$. Dividing by $n\vol{T_\phi}$ and passing to the limit we get 
\[
1-\xi\kappa L(\phi)\le\liminf\frac{\vol{Q(o_1,\phi^no_1)}}{n\vol{T_\phi}}\le\limsup\frac{\vol{Q(o_1,\phi^no_1)}}{n\vol{T_\phi}}\le 1+\xi\kappa L(\phi).
\]
As $\xi$ is arbitrary, the claim for $o_1$ follows. For a general $o$, it suffices to notice that, by Proposition \ref{replacement}, the difference $\left|\vol{Q(o,\phi^no)}-\vol{Q(o_1,\phi^no_1)}\right|$ is uniformly bounded by $\kappa(d_\T(o,o_1)+d_\T(\phi^no,\phi^no_1))+\kappa=2\kappa d_\T(o,o_1)+\kappa$.
\end{proof}

We remark that the results mentioned above \cite{BB}, \cite{KoMc} prove something stronger, that is $\left|2n\vol{T_\phi}-\vol{Q(\phi^{-n}o,\phi^no)}\right|=O(1)$.

\section{Random Walks}
\label{random walks}

We start talking about random walks on the mapping class group. We set up terminology, notations and first observations. The goal of the section is to introduce the third and last major technical tool of the paper which is a recurrence property (Proposition \ref{recurrence}).

\subsection{Random walks on the mapping class group}
We will work in the following generalities:

\begin{center}
\label{notation2}
\begin{minipage}{.8\linewidth}
\textbf{Standing assumption}. Let $S\subset\mcg$ be a finite symmetric set $S=S^{-1}$ generating the group $G=\langle S\rangle$. Let $\mu$ be a probability measure whose support equals $S$. We only consider random walks driven by probability measures arising this way with $G=\text{\rm Mod}(\Sigma)$.
\end{minipage}
\end{center}

Let us start with the most basic definition:

\begin{dfn}[Random Walk]
A \emph{random walk on $G$ driven by $\mu$} is given by the following data: Let $\ga{s_n}_{n\in\mb{N}}$ be a sequence of random variables with values in $S$ which are independent and have the same distribution $\mu$. The $n$-\emph{th step of the random walk} is the random variable $\omega_n:=s_1\dots s_n$. The random walk is the process $\omega:=\pa{\omega_n}_{n\in\mb{N}}$.

\begin{center}
\label{notation2}
\begin{minipage}{.8\linewidth}
\textbf{Notation}. We will always denote by $s=\pa{s_n}_{n\in\mb{N}}$ the sequence of labels and by $\omega=\pa{\omega_n:=s_1\dots s_n}_{n\in\mb{N}}$ the path traced by the sequence of labels.
\end{minipage}
\end{center}

The distribution of the $n$-th step of the random walk coincides with the $n$-th fold convolution $\mb{P}_n$ of $\mu$ with itself. It is given inductively by:
\[
\mb{P}_n\qa{E}:=\sum_{s\in S}{\mu(s)\mb{P}_{n-1}\qa{s^{-1}E}}.
\]
Let $\mc{P}$ be a property of mapping classes $f\in\mcg$. We call it {\em typical} if it is very likely that a random mapping class has it, that is
\[
\mb{P}_n\qa{f\in\mcg\left|\text{ $f$ has $\mc{P}$}\right.}\stackrel{n\rightarrow\infty}{\longrightarrow}1.
\] 
\end{dfn} 

The starting point of our discussion are two results by Maher \cite{Ma}, \cite{Ma10} that ensure that the property ``$X_f$ is hyperbolic'' is typical and hence it makes sense to consider the hyperbolic volume of $X_f$. 

\begin{dfn}[Sample Paths]
The {\em space of sample paths} is the measurable space $(\Omega,\mc{E})$ where $\Omega:=G^{\mb{N}}$ and $\mc{E}$ is the $\sigma$-algebra generated by the {\em cylinder sets}. Given a probability measure $\mu$ on $G$, we get a {\em probability measure} $\mb{P}$ on $\Omega$ induced by the random walk driven by $\mu$. It is the push-forward $\mb{P}:=T_*\mu^{\mb{N}}$ of the product measure $\mu^{\mb{N}}$ under the following measurable transformation:
\[
T:\Omega\rightarrow\Omega\quad\text{ defined by }\quad T(s)=\omega.
\]
\end{dfn}

\begin{dfn}[Shift Operator]
On the space of sample paths $\Omega$ there is a natural \emph{shift operator} $\sigma:\Omega\rightarrow\Omega$ defined by 
\[
\pa{\sigma\pa{s_i}_{i\in\mb{N}}}_j=s_{j+1}.
\]
If $\omega=T(s)=(\omega_n=s_1\dots s_n)_{n\in\mb{N}}\in\Omega$ is the path traced by a random walk, then we can write $(\sigma^i\omega)_j=\omega_i^{-1}\omega_{i+j}$. It is a standard computation on cylinder sets to check that $\sigma$ preserves $\mu^{\mb{N}}$ and that $(\Omega,\mu^{\mb{N}},\sigma)$ is mixing and hence {\em ergodic}.  
\end{dfn}

\subsection{Linear drift and sublinear tracking}
Consider the action on Teichm\"uller space $G\curvearrowright\T$ and fix a basepoint $o\in\T$. Every random walk $\omega=\pa{\omega_n}_{n\in\mb{N}}\in\Omega$ traces an orbit $\{\omega_no\}_{n\in\mb{N}}\subset\T$.

It follows from the triangle inequality that the random variables $d_{\T}(o,\omega_no)$ are subadditive with respect to the shift map $\sigma$. By Kingman's subadditive ergodic theorem and ergodicity of $(\Omega,\mb{P},\sigma)$, there exists a constant $L_\T\ge 0$, called the \emph{drift} of the random walk on Teichm\"uller space, such that for $\mb{P}$-almost every sample path $\omega\in\Omega$ we have
\[
\frac{d_\T\pa{o,\omega_no}}{n}\stackrel{n\rightarrow\infty}{\longrightarrow}L_\T.
\]

It is natural to ask whether the orbit $\{\omega_no\}_{n\in\mb{N}}$ converges to some point on the Thurston compactification of Teichm\"uller space $\mc{PML}$. This property was first established by Kaimanovich-Masur \cite{KM}.

\begin{thm}[Kaimanovich-Masur \cite{KM}]
\label{poisson}
We have $L_\T>0$. For $\mb{P}$-almost every sample path $\omega=(\omega_n)_{n\in\mb{N}}\in\Omega$ and for every basepoint $o\in\T$, the sequence $\{\omega_no\}_{n\in\mb{N}}$ converges to a point $\text{\rm bnd}(\omega)\in\mc{PML}$ which is independent of $o\in\T$. The map $\text{\rm bnd}:\Omega\rightarrow\mc{PML}$ is Borel measurable. Moreover, $\mb{P}$-almost surely, the point $\text{\rm bnd}(\omega)$ is uniquely ergodic, minimal and filling. 
\end{thm}

Moreover, Tiozzo \cite{T15} showed that the orbit $\{\omega_no\}_{n\in\mb{N}}$ can also be \emph{tracked} by a Teichm\"uller ray in the following sense:

\begin{thm}[Tiozzo \cite{T15}]
\label{tracking}
For $\mb{P}$-almost every sample path $\omega=(\omega_n)_{n\in\mb{N}}\in\Omega$ and for every basepoint $o\in\T$, there exists a unit speed Teichm\"uller ray $\tau:[0,+\infty)$ starting at $\tau(0)=o$ and ending at $\tau(\infty)=\text{\rm bnd}(\omega)$ such that
\[
\lim_{n\rightarrow\infty}\frac{d_\T\pa{\omega_no,\tau(L_\T n)}}{n}=0.
\]
\end{thm}

\subsection{Recurrence}
Now we can present our last fundamental ingredient which is the following recurrence property:   

\begin{thm}[Baik-Gekhtman-Hamenst\"adt, Propositions 6.9 and 6.11 of \cite{BGH16}]
\label{recurrence}
Let $o\in\T$ be a basepoint and $\tau_\omega$ the tracking ray for $\omega$. Then:
\begin{itemize}
\item{{\rm Recurrence}: For every $\eta>0$ sufficiently small, for every $0<a<b$ and $h>0$, for $\mb{P}$-almost every $\omega$ with tracking ray $\tau_\omega$ there exists $N=N(\omega)>0$ such that for every $n\ge N$ the segment $\tau_\omega\qa{an,bn}$ has a connected subsegment of length $h$ entirely contained in $\T_\eta$.}
\item{{\rm Fellow-Traveling}: There exists $\delta>0$ such that for every $\ep>0$ and for $\mb{P}$-almost every sample path $\omega$ there exists $N=N(\omega)>0$ such that for every $n\ge N$, the element $\omega_n$ is pseudo-Anosov with translation length $L(\omega_n)\in\qa{(1-\ep)L_\T n,(1+\ep)L_\T n}$. Its axis $l_n$ $\delta$-fellow-travels the tracking ray $\tau_\omega$ on $\qa{\ep L_\T n,(1-\ep)L_\T n}$, i.e. for every $t\in\qa{\ep L_\T n,(1-\ep)L_\T n}$ we have $d_\T\pa{\tau_\omega(t),l_n}<\delta$.}
\end{itemize} 
\end{thm}

For the convergence $L(\omega_n)/n\rightarrow L_\T$ see also Dahmani-Horbez \cite{DH18}.

\subsection{A larger class of random walks}
As stated at the beginning of the section, in this paper we only work with probability measures $\mu$ with finite support $S$ that generates the full mapping class group $G=\text{\rm Mod}(\Sigma)$. This allows us to keep the statements uniform and to avoid distinguishing between different families of random 3-manifolds. 

However, at the price of making a distinction between mapping tori, quasi-fuchsian manifolds and Heegaard splittings, the assumptions on $\mu$ can be considerably relaxed and still obtain the convergence results in Theorems \ref{main1} and \ref{qf tracking}. We briefly describe, without details, two larger classes of random walks to which our results can be extended.

For mapping tori and quasi-fuchsian manifolds it is enough that $S$, the finite support of $\mu$, generates a subgroup $G$ containing two pseudo-Anosov elements that act as independent loxodromics on the {\em curve graph} (see \cite{MT14} for the definitions). All the theorems in this section hold in these generalities.

For Heegaard splittings, we further require that the two pseudo-Anosov elements also act as independent loxodromics on the {\em handlebody graph} (see \cite{MS18} for a definition). Crucially, the condition implies, by work of Maher-Schleimer \cite{MS18} and Maher-Tiozzo \cite{MT14}, that random walks on $G$ have a {\em positive drift} on the handlebody graph. This ensures that a random Heegaard splitting is hyperbolic and plays a role also in the construction of the model metric from \cite{HV} used in the next section. 

With these caveats, the proofs can be extended by following word-by-word the same lines, no change is needed.

\section{A Law of Large Numbers for the Volume}
\label{lln volume}

We are ready to prove the law of large numbers for the volumes of random 3-manifolds.

\begin{mythm}[1]
$\mb{P}$-almost surely the limit following limit exists
\[
\lim_{n\rightarrow\infty}{\frac{\vol{X_{\omega_n}}}{n}}=v.
\]
The family of 3-manifold $\ga{X_{\omega_n}}_{n\in\mb{N}}$ can denote either the mapping tori or the Heegaard splittings defined by $\omega_n$.
\end{mythm}

We will deduce it from the following analogue concerning quasi-fuchsian manifolds. The idea is that, according to the geometric models, the volume of a random 3-manifold is always captured by a quasi-fucshian manifold.

\begin{mythm}[2]
For every $o\in\T$ and for $\mb{P}$-almost every $\omega\in\Omega$ the following limit exists:
\[
\lim_{n\rightarrow\infty}{\frac{\vol{Q(o,\omega_no)}}{n}}=v.
\]
\end{mythm}

Let us remark again that $v=v(\mu)>0$ is the same as in Theorem \ref{main1}.

\subsection{Mapping tori and Heegaard splittings}
Let us assume Theorem \ref{qf tracking} and prove the result for random 3-manifolds:

\begin{proof}[Proof of Theorem \ref{main1}]
Fix $\ep>0$. Let $\tau_\omega:[0,\infty)\rightarrow\T$ be the ray connecting $o$ to $\text{\rm bnd}(\omega)$.

{\bf Mapping tori}. We use the model for $T_{\omega_n}$ coming from Corollary \ref{qf vs mt} (see also Figure \ref{fig:figure7}): By Proposition \ref{recurrence}, if $n$ is large enough, we can find on $\tau_\omega$ four points $x_n<y_n<z_n<w_n<x_n+L(\omega_n)$ such that the intervals $[x_n,y_n]$ and $[z_n,w_n]$ satisfy the hypotheses of Corollary \ref{qf vs mt}: They are contained in $[\ep L_\T n,2\ep L_\T n]$ and $[(1-2\ep)L_\T n,(1-\ep)L_\T n]$ respectively. Their length is at least $h$ and their image is $\eta$-thick. The restriction of $\tau_\omega$ to $[x_n,w_n]$ $\delta$-fellow travels the Teichmüller axis $l_n:\mb{R}\rightarrow\T$ of $\omega_n$ whose translation length is roughly $(1-\ep)L_\T n\le L(\omega_n)\le (1+\ep)L_\T n$. Applying Corollary \ref{qf vs mt} we get:

\begin{lem}
\label{interpolation}
For $\mb{P}$-almost every $\omega$ and every large enough $n\ge n_\omega$ we have
\[
\left|\vol{Q(\tau_\omega(x_n),\tau_\omega(w_n))}-\vol{T_{\omega_n}}\right|\le\ep n.
\]
and
\[
\left|\vol{Q_{\omega_n}}-\vol{Q(\tau_\omega(x_n),\tau_\omega(w_n))}\right|\le\ep n.
\]
\end{lem}

\begin{proof}[Proof of Lemma \ref{interpolation}]
By Corollary \ref{qf vs mt} we have
\begin{align*}
&\left|\vol{T_{\omega_n}}-\vol{Q(l_n(x_n),l_n(w_n))}\right|\\
&\le\kappa(L(\omega_n)+y_n-z_n)+\xi\kappa(w_n-x_n)+\text{\rm const}\\
&\le\kappa 4\ep L_\T n+\xi\kappa(1-2\ep)L_\T n+\text{\rm const}. 
\end{align*}
Up to a uniform additive constant we can also replace $Q(l_n(x_n),l_n(w_n))$ with $Q(\tau_\omega(x_n),\tau_\omega(w_n))$. If $n$ is large enough we can improve the last quantity to $\ep n$. Instead, from Proposition \ref{replacement}
\begin{align*}
&\left|\vol{Q(o,\omega_no)}-\vol{Q(\tau_\omega(x_n),\tau_\omega(w_n))}\right|\\
&\le \kappa(d_\T(o,\tau_\omega(x_n))+d_\T(\tau_\omega(w_n),\omega_no))+\kappa\\
&\le\kappa(d_\T(o,\tau_\omega(x_n))+d_\T(\tau_\omega(w_n),\tau_\omega(L_\T n))+d_\T(\tau_\omega(L_\T n),\omega_no))+\kappa.
\end{align*}
By our choice of $x_n,w_n$ and Tiozzo's sublinear tracking (Theorem \ref{tracking}), if $n$ is large enough, we can bound the last quantity by $\ep n$.
\end{proof}

Lemma \ref{interpolation} and Theorem \ref{qf tracking} imply that $|\vol{T_{\omega_n}}-nv|=o(n)$ which concludes the proof for mapping tori.

\begin{figure}
\begin{center}
\begin{overpic}{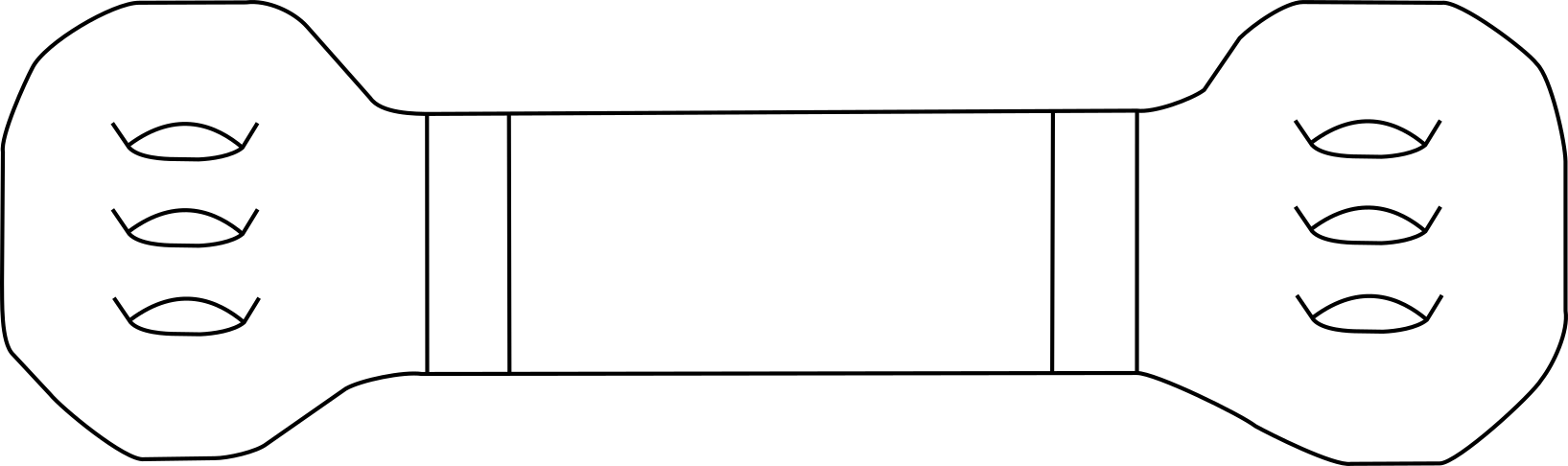}
\put (1,12) {$H_1$}
\put (28,12) {$\Omega_1$}
\put (47,12) {$Q^0$}
\put (68,12) {$\Omega_2$}
\put (93,12) {$H_2$}
\end{overpic}
\end{center}
\caption{Model for a random Heegaard splitting.}
\label{fig:figure6}
\end{figure}
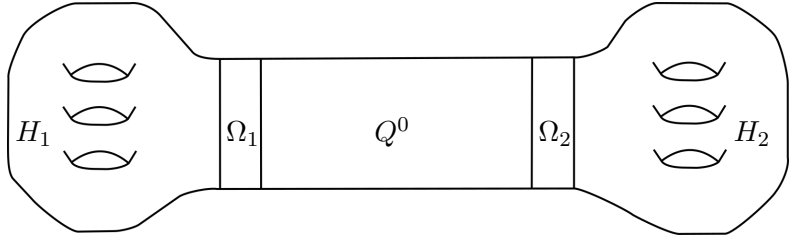

{\bf Heegaard splittings}. The argument is completely analogous to the previous one, but the model is different. We use the one constructed in \cite{HV}, in particular Proposition 7.1. For convenience of the reader we give a brief description of it: Recall that $\ep>0$ is fixed. A random Heegaard splitting $M_{\omega_n}$ admits a negatively curved Riemannian metric $\rho$ with the following properties (see Figure \ref{fig:figure6}): It is purely hyperbolic outside two regions $\Omega:=\Omega_1\sqcup\Omega_2$ which have uniformly bounded diameter and where the sectional curvatures lie in the interval $(-1-\ep,-1+\ep)$. The complement $M_{\omega_n}-\Omega$ decomposes into three connected pieces $H_1\sqcup Q^0\sqcup H_2$. The pieces $H_1,H_2$ are homeomorphic to handlebodies and have small volume $\text{\rm vol}(H_1\sqcup H_2\sqcup\Omega)\le\ep n$. The middle piece $Q^0$ embeds isometrically in the convex core of $Q(o,\omega_no)$, moreover $\text{\rm vol}(Q(o,\omega_no))-\text{\rm vol}(Q^0)\le\ep n$. Hence we can apply again Theorem \ref{bcg} and Theorem \ref{qf tracking}.     
\end{proof}

We now proceed with the proof of Theorem \ref{qf tracking}.

\subsection{Strategy overview}
Denote by $Q_\phi$ the manifold $Q(o,\phi o)$. 

We want to show that for $\mb{P}$-almost every $\omega$ the sequence $\text{\rm vol}(Q_{\omega_n})/n$ converges. Suppose this is not the case. Then there exists a set $\Omega_{\text{\rm bad}}$ with positive measure $\mb{P}\qa{\Omega_{\text{\rm bad}}}>0$ on which 
\[
\limsup_{n\rightarrow\infty}{\frac{\vol{Q_{\omega_n}}}{n}}-\liminf_{n\rightarrow\infty}{\frac{\vol{Q_{\omega_n}}}{n}}>0.
\]
We can as well assume that there is a small $\ep_0>0$ and a set $\Omega_{\text{\rm bad}}^{\ep_0}$ with positive measure $\zeta_0:=\mb{P}\qa{\Omega_{\text{\rm bad}}^{\ep_0}}>0$ on which the difference is at least $\ep_0>0$. Hence, in order to get a contradiction, it is enough to prove that for every $\ep,\zeta>0$ there exists a set $\Omega_{\ep,\zeta}$ with measure $\mb{P}\qa{\Omega_{\ep,\zeta}}\ge 1-\zeta$ on which the difference between limsup and liminf is smaller than $\ep$.

First we observe that we can exploit a {\em neighbour approximation property} of the volumes (Lemma \ref{nearby points}). It allows a convenient technical reduction: We can make the random walk {\em faster} and still keep under control the asymptotic behaviour (Lemma \ref{real numbers}). The faster we make the random walk the more properties we can prescribe, a feature that will be important in Proposition \ref{approximation}. The central step of the proof consists of finding a set on which the variables $\text{\rm vol}(Q_{\omega_{nN}})$ and the {\em ergodic sum} $\sum_{j<n}{\text{\rm vol}(Q_{\sigma^{jN}(\omega)_N})}$ are comparable (Proposition \ref{approximation}). Finally, we use the {\em ergodic theorem} to conclude the proof.

\subsection{A faster random walk}
For every $N\in\mb{N}$ we can replace the random walk $\omega$ with $(\omega_{jN})_{j\in\mb{N}}$ and the shift map $\sigma$ with $\sigma^N$. The dynamical system $(\Omega,\mu^{\mb{N}},\sigma^N)$ is still ergodic. As we wish to apply the ergodic theorem, we discuss the integrability condition of the volume function and the relations between the asymptotics of the faster random walk and the original one. Recall that $S$, the support of $\mu$, is symmetric and generates $G=\text{\rm Mod}(\Sigma)$. 

\begin{lem}
\label{integrable}
There exists $C>0$ such that for every $\phi\in G$ we have $\vol{Q_{\phi}}\le C\left|\phi\right|_S+C$ where $\left|\phi\right|_S$ is the word length in the generating set $S$.
\end{lem}

\begin{proof}
Let $\phi=s_1\dots s_n$ with $s_i\in S$. By Proposition \ref{replacement} we have $\vol{Q_{\phi}}\le\kappa d_\T(o,\phi o)+\kappa$. By the triangle inequality $d_\T(o,s_1\dots s_no)\le\sum_{j<n}{d_\T(o,s_jo)}$ $\le\max_{s\in S}\ga{d_\T(o,so)}n$.
\end{proof}

In particular, for any fixed $n\in\mb{N}$, the function $\vol{Q_{\omega_n}}$ is integrable on $(\Omega,\mc{E},\mb{P})$ and we can apply the Birkhoff ergodic theorem. Moreover, we have the following neighbour approximation property.

\begin{lem}
\label{nearby points}
For $\mb{P}$-almost every sample path $\omega\in\Omega$, for every $n,m$ we have
\[
\left|\vol{Q_{\omega_{n+m}}}-\vol{Q_{\omega_n}}\right|\le Cm+C.
\]
\end{lem}

\begin{proof}
By Proposition \ref{replacement} $\left|\vol{Q_{\omega_{n+m}}}-\vol{Q_{\omega_n}}\right|\le\kappa d_\T(\omega_no,\omega_{n+m}o)+\kappa$. From the triangle inequality $d_\T(\omega_no,\omega_{n+m}o)\le C\left|\omega_n^{-1}\omega_{n+m}\right|_S\le Cm$. 
\end{proof}

The next completely elementary lemma illustrates why the neighbour approximation property allows to speed up the random walk without loosing control on the asymptotic behaviour.

\begin{lem}
\label{real numbers}
Consider a sequence $\ga{a_n}_{n\in\mb{N}}\subset\mb{R}$ and an integer $N\in\mb{N}$. Suppose that the sequence satisfies $\left|a_{n+m}-a_n\right|\le Cm+C$ for every $n,m$. Assume that $A:=\limsup_{j\rightarrow\infty}{\frac{a_{jN}}{jN}}$ and $a:=\liminf_{j\rightarrow\infty}{\frac{a_{jN}}{jN}}$ are finite. Then
\[
a\le\liminf_{n\rightarrow\infty}{\frac{a_n}{n}}\le\limsup_{n\rightarrow\infty}{\frac{a_n}{n}}\le A.
\]
\end{lem}

\subsection{Comparison with ergodic sums}

The following is our main estimate

\begin{pro}
\label{approximation}
Fix $\ep,\zeta>0$. There exists $N(\ep,\zeta)>0$ and a set $\Omega_{\ep,\zeta,N}$ with $\mb{P}\qa{\Omega_{\ep,\zeta,N}}\ge 1-\zeta$ such that for every $\omega\in\Omega_{\ep,\zeta,N}$ and $n\in\mb{N}$ large enough we have 
\[
\left|\vol{Q_{\omega_{nN}}}-\sum_{0\le j<n}{\vol{Q_{(\sigma^{jN}\omega)_N}}}\right|\le\text{\rm const}\cdot\ep nN
\]
for some uniform $\text{\rm const}>0$. 
\end{pro}

We will show that, for a suitably chosen $N$, both families $\{Q_{\omega_{nN}}\}$ and $\{Q_{(\sigma^{jN}\omega)_N}\}_{j<n}$ can be {\em refined} to construct {\em models}, via Proposition \ref{gluing}, for the hyperbolic mapping torus $T_{\omega_{nN}}$. The central property of the models is that they {\em nearly compute the volume} $\vol{T_{\omega_{nN}}}$. This suffices to conclude.

\begin{proof}
Let $\delta>0$ be the fellow traveling constant of Proposition \ref{recurrence} and $h$ a large height. Since the value of $L_\T>0$ is irrelevant and only complicates some formulas below by affecting the value of some constants, we are going to assume $L_\T=1$. In the course of the proof, specifically in the inequalities (1)-(13), we will get several uniform constants which depend on previous steps and whose explicit expressions are irrelevant for the argument. In order to simplify the exposition we will always denote these different constants by $\text{\rm const}>0$. 

For every $N$ denote by $\Omega_{\ep,N}$ the set of paths satisfying the following properties
\begin{enumerate}
\item{$\omega_n$ is pseudo-Anosov and $L(\omega_n)/n\in(1-\ep,1+\ep)$ for every $n\ge N$.}
\item{$l_n$, the axis of $\omega_n$, $\delta$-fellow travels $\tau_\omega[\ep n,(1-\ep)n]$ for every $n\ge N$.}
\item{$\omega_n\tau_\omega[\ep n,\infty]$ $\delta$-fellow travels $\tau_\omega[(1+\ep)n,\infty]$ for every $n\ge N$.}
\item{$\tau_\omega[\ep n,2\ep n]$ and $\tau_\omega[(1\pm\ep)n,(1\pm 2\ep)n]$ contain $\eta$-thick subsegments of length at least $h$ for every $n\ge N$.}
\item{The conclusions of Lemma \ref{interpolation} hold for every $n\ge N$.}
\item{$d_\T(o,\omega_no)/n\in(1-\ep,1+\ep)$ for all $n\ge N$.}
\end{enumerate}

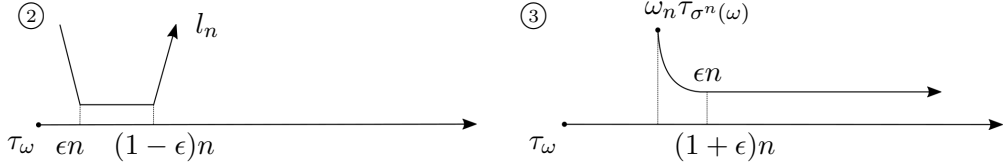
\begin{figure}[H]
\begin{center}
\begin{overpic}{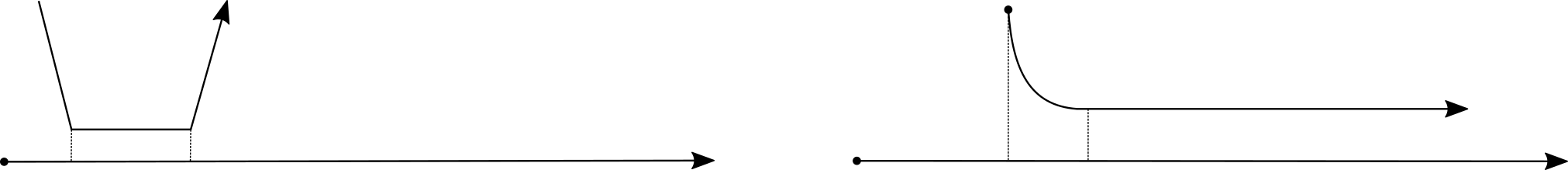}
\put (-3,-2) {$\tau_\omega$}
\put (2,-2.5) {$\ep n$}
\put (8,-2.5) {$(1-\ep)n$}
\put (51,-2) {$\tau_\omega$}
\put (-2,11) {\textcircled{\tiny 2}}
\put (50,11) {\textcircled{\tiny 3}}
\put (63,12) {$\omega_n\tau_{\sigma^n(\omega)}$}
\put (66,-2.5) {$(1+\ep)n$}
\put (68,5) {$\ep n$}
\put (16.5,10) {$l_n$}
\end{overpic}
\end{center}
\caption{Properties 2 and 3.}
\label{fig:figure3}
\end{figure}

Observe that if $N_1\ge N_2$ then $\Omega_{\ep,N_2}\subseteq\Omega_{\ep,N_1}$, if we enlarge $N$ the set can only get bigger. We reserve ourselves the right to determine later suitably modified constants $\delta,h,N$. Since all the properties are satisfied asymptotically with probability one, for fixed $\ep,\zeta>0$ there exists some $N(\ep,\zeta,h)$ such that $\Omega_{\ep,N}$ has measure at least $1-\zeta$. Fix $N$ larger than this threshold and speed up the random walk, that is replace $\omega$ with $(\omega_{jN})_{j\in\mb{N}}$ and $\sigma$ with $\sigma^N$.

By ergodicity of $(\Omega,\mu^{\mb{N}},\sigma^N)$, the orbits $\{\sigma^{jN}\omega\}_{j\in\mb{N}}$ will visit the set $\Omega_{\ep,N}$ very often, the number of hitting times being proportional to the measure of the set $\ge 1-\zeta$. We record the hitting times by subdividing the interval $[n]=\{0,\dots, n\}$ into a disjoint union of consecutive intervals $[n]=I_1\sqcup J_1\sqcup\dots\sqcup I_k\sqcup J_k$ where the $I_i$'s contain the indices $j$ for which $\sigma^{jN}\omega\in\Omega_{\ep,N}$, whereas the $J_i$'s are the bad indices ($J_k$ might be empty). By the ergodic theorem the total length of the bad intervals is controlled by
\[
\frac{1}{n}\sum_{j<n}{\mathds{1}_{\Omega\setminus\Omega_{\ep,N}}(\sigma^{jN}\omega)}=\frac{1}{n}\sum_{i\le k}{|J_i|}\stackrel{n\rightarrow\infty}{\longrightarrow}\mb{P}\qa{\Omega\setminus\Omega_{\ep,N}}\le\zeta.
\] 

{\bf Basic case}. We start by proving the proposition assuming that all indices are good. Since our considerations are all independent of the past, we will also get a ``local version'' of the proposition for every good interval $I_j$.  

We are going to define two families of quasi-fuchsian manifolds that satisfy the hypotheses of Proposition \ref{gluing} and can be glued to form a model for $T_{\omega_{nN}}$ that nearly computes its volume. The two families consist of:
\begin{enumerate}[{\bf I}]
\item{Quasi-fuchsian manifolds related to $Q_{\sigma^{jN}(\omega)_N}$ for every $j\in [n]$.}
\item{A single quasi-fuchsian manifold related to $Q_{\omega_{nN}}$ as in Lemma \ref{interpolation}.}
\end{enumerate} 

\begin{figure}[H]
\begin{center}
\begin{overpic}{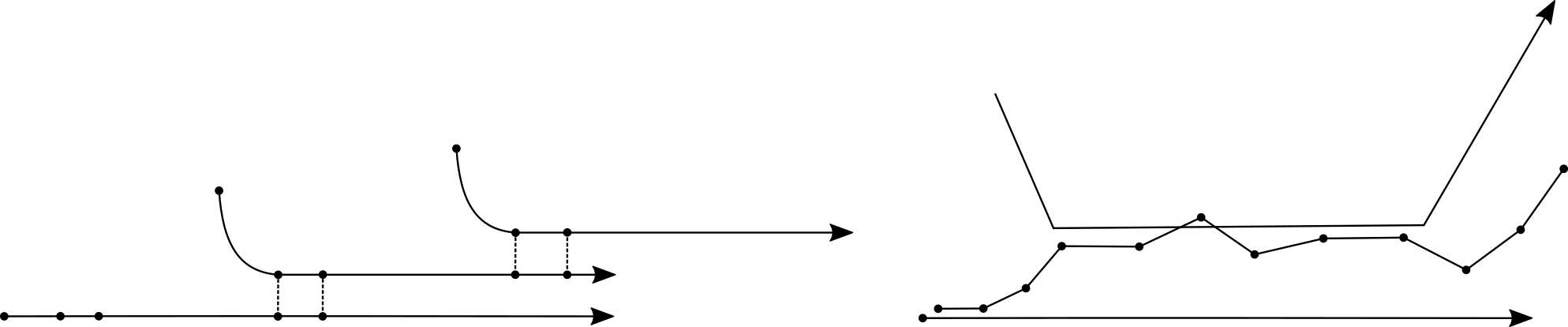}
\put (-2,18) {\textcircled{\tiny A}}
\put (56,18) {\textcircled{\tiny B}}
\put (0,2) {$\tau_\omega$}
\put (2.5,-2.5) {$a_0$}
\put (6.5,-2.5) {$b_0$}
\put (16,-2.5) {$c_0$}
\put (20,-2.5) {$d_0$}
\put (16,4.5) {$a_1$}
\put (20,4.5) {$b_1$}
\put (4,11) {$\omega_N\tau_{\sigma^N(\omega)}$}
\put (22,13) {$\omega_{2N}\tau_{\sigma^{2N}(\omega)}$}
\put (60,13) {$l_n$}
\put (56,-2) {$\tau_\omega$}
\put (68,7.5) {$a_r$}
\put (72,7.5) {$d_r$}
\put (84,7.5) {$a_s$}
\put (88,7.5) {$d_s$}
\end{overpic}
\end{center}
\caption{Basic case.}
\label{fig:figure4}
\end{figure}
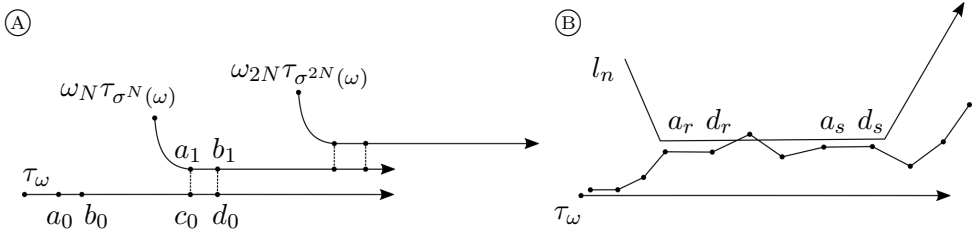

{\bf Family I}. Proceed inductively. Begin with $i=0$ and the two Teichmüller rays $\tau_\omega$ and $\omega_N\tau_{\sigma^N(\omega)}$. The restrictions $\omega_N\tau_{\sigma^N(\omega)}|_{[\ep N,\infty)}$ and $\tau_\omega|_{[(1+\ep)N,\infty)}$ are $\delta$-fellow travelers. The ray $\tau_\omega$ contains four points $a_0<b_0<c_0<d_0$ such that $[a_0,b_0]\subset[\ep N,2\ep N]$ and $[c_0,d_0]\subset[(1+\ep)N,(1+2\ep)N]$, their image is $\eta$-thick and their length is at least $h$ (see Figure \ref{fig:figure4} A). The segment $[c_0,d_0]$ determines $[a_1,b_1]$ by the condition that $\omega_N\tau_{\sigma^N(\omega)}[a_1,b_1]$ $\delta$-fellow travels $\tau_\omega[a_0,b_0]$ and $[a_1,b_1]\subset[\ep N,2\ep N]$. As $1\in [n]$ is good, we can go on and find $[c_1,d_1]\subset[(1+\ep)N,(1+2\ep)N]$ of length at least $h$ and with $\tau_{\sigma^N(\omega)}$-image in $\T_\eta$. Inductively we determine $a_i<b_i<c_i<d_i$ for every $i\le n$. Before going on, let us simplify a little the notation by introducing
\begin{align*}
&A_i=\omega_{iN}\tau_{\sigma^{iN}(\omega)}(a_i), &B_i=\omega_{iN}\tau_{\sigma^{iN}(\omega)}(b_i),\\
&C_i=\omega_{iN}\tau_{\sigma^{iN}(\omega)}(c_i), &D_i=\omega_{iN}\tau_{\sigma^{iN}(\omega)}(d_i).
\end{align*}
We associate to the index $i\le n$ the quasi-fuchsian manifold $Q(A_i,D_i)$. Informally, we renormalized the picture by placing ourselves at the $iN$-th point of the orbit $O_i=\omega_{iN}o$. From there we see the segment $[A_i,D_i]$ that $\delta$-fellow travels $[O_i,\text{\rm bnd}(\omega)]$. Observe that, by Proposition \ref{replacement}, we have 
\begin{align}
\left|\vol{Q(A_i,D_i)}-\vol{Q_{\sigma^{iN}(\omega)_N}}\right| \label{ineq1}\\ 
\le\kappa(d_\T(O_i,A_i)+d_\T(D_i,O_{i+1}))+\kappa\le\kappa4\ep N+\text{\rm const}.\nonumber
\end{align}

Sequences of consecutive good indices are geometrically controlled: 

\begin{lem}
\label{closing}
The segment $[A_i,D_i]$ uniformly fellow travels $[O,O_n]$.
\end{lem} 

\begin{proof}
Let $\mc{C}$ be the curve graph of $\Sigma$. Consider the shortest curve projection $\Upsilon:\T\rightarrow\mc{C}$. By Masur-Minsky \cite{MM99} we have the following: The curve graph $\mc{C}$ is hyperbolic and the projection is uniformly coarsely Lipschitz and sends Teichmüller geodesics to {\em unparametrized} uniform quasi geodesics. In particular, by stability of quasi geodesics, $\Upsilon[A_i,D_i]$ is uniformly Hausdorff close to the geodesic segment $[\Upsilon(A_i),\Upsilon(C_i)]$. The same holds true for $\Upsilon[O,O_n]$ and $[\Upsilon(O),\Upsilon(O_n)]$. 

Since the composition of $\Upsilon$ with a parametrized, $\eta$-thick and sufficiently long Teichmüller geodesic is a uniform {\em parametrized} quasi geodesic (see \cite{H10}), we also have the following: If the $\delta$-fellow traveling $h$ between $[C_{i-1},D_{i-1}]$ and $[A_i,B_i]$ is sufficiently long, then the geodesics $[\Upsilon(A_{i-1}),\Upsilon(D_{i-1})]$ and $[\Upsilon(A_i),\Upsilon(D_i)]$ uniformly fellow travel along a segment, terminal for the first and initial for the second, which is as long as we wish.

In particular this implies that, if $h$ is large enough, then the concatenation of the geodesic segments
\[
[\Upsilon(O),\Upsilon(C_0)]\cup[\Upsilon(A_1),\Upsilon(C_1)]\cup\dots\cup[\Upsilon(A_{n-1}),\Upsilon(C_{n-1})]\cup[\Upsilon(A_n),\Upsilon(O_n)]
\]
is a uniform $(1,K)$ local quasi geodesic. By the stability of uniform local quasi geodesics in hyperbolic spaces, we conclude that every segment $[\Upsilon(A_i),\Upsilon(D_i)]$ lies uniformly Hausdorff close to $[\Upsilon(O),\Upsilon(O_n)]$. 

In particular, there are points $P_i,Q_i\in[O,O_n]$ for which the projection is uniformly close to the projections of $[A_i,B_i]$ and $[C_i,D_i]$. As Teichmüller geodesics in the thick part are uniformly contracting (by \cite{M96} and \cite{H10}) it follows that $P_i,Q_i$ are uniformly close to the thick subsegments of $[A_i,B_i]$, $[C_i,D_i]$. Therefore, by \cite{R14}, $[P_i,Q_i]$ uniformly fellow travels $[A_i,D_i]$ provided that the height $h$ is sufficiently large. 
\end{proof}

Observe that, by property (2), the segment $[O,O_n]$ uniformly fellow-travels the axis $l_n$ of the pseudo-Anosov $\omega_{nN}$ along the subsegment $[\ep Nn,(1-\ep)Nn]$. By Lemma \ref{closing}, there is a subsegment $[r,s]\subset [n]$ of size $s-r\ge(1-\ep)n$, obtained by discarding an initial and a terminal subsegment of length proportional to $\ep n$, such that for all $r\le i\le s$ $[A_i,D_i]$ uniformly fellow travels $l_n$ (see Figure \ref{fig:figure4} B). We add to the collection the quasi-fuchsian manifold $Q(C_s,\omega_{nN}B_r)$. Using Proposition \ref{replacement} we see that 
\begin{align}
\vol{Q(C_s,\omega_{nN}B_r)}\le\kappa d_\T(C_s,\omega_{nN} B_r)+\kappa\le\text{\rm const}\cdot \ep nN. \label{ineq2}
\end{align}
In fact, on the one hand, the points $B_r,C_s$ are, respectively, uniformly close to points $l_n(t_r),l_n(t_s)$ so their distance is roughly $t_s-t_r$ and $d_\T(C_s,\omega_{nN}B_r)$ can be bounded by $L(\omega_{nN})-(t_s-t_r)$. On the other hand, combining property (6) and $s-r\ge(1-\ep)n$, their distance, up to an error of $\ep N$, is also given by $d_\T(O_r,O_s)\ge(1-\ep)(s-r)N$. By property (1) we have $L(\omega_{nN})\le(1+\ep)nN$ so $L(\omega_{nN})-(t_s-t_r)\simeq(1+\ep)nN-(1-\ep)^2nN$ whence inequality (2). 

Moreover, by Lemma \ref{integrable} and the fact that $|[n]\setminus[r,s]|\le\ep n$, we have
\begin{align}
\sum_{j\not\in[r,s]}{\vol{Q_{(\sigma^{jN}\omega)_N}}}\le\sum_{j\not\in[r,s]}{CN+C}\le\text{\rm const}\cdot\ep nN. \label{ineq3}
\end{align}
By construction, the family $\{Q(A_i,D_i)\}_{r\le i\le s}\sqcup\{Q(C_s,\omega_{nN}B_r)\}$ satisfies the gluing conditions of Proposition \ref{gluing} provided that $h$ is very large. As a result
\begin{align}
\left|\vol{T_{\omega_{nN}}}-\sum_{i\in[r,s]}{\vol{Q(A_i,D_i)}}-\vol{Q(C_s,\omega_{nN}B_r)}\right|\label{ineq4}\\
\le nV_0+\text{\rm const}\cdot\ep nN \nonumber
\end{align}
where $V_0=V_0(\eta,\xi,h,D_1)$ is as in Proposition \ref{gluing}. 

{\bf Family II}. By property (5) and Lemma \ref{interpolation}, we can find on $\tau_\omega$ a pair of points $x_n\in[\ep nN,2\ep nN]$ and $w_n\in[(1-2\ep)nN,(1-\ep)nN]$ which define a quasi-fuchsian manifold whose volume approximate simultaneously the volume of the mapping torus $T_{\omega_{nN}}$ and the volume of the quasi-fuchsian manifold $Q_{\omega_{nN}}$
\begin{align}
\left|\vol{T_{\omega_{nN}}}-\vol{Q(\tau_\omega(x_n),\tau_\omega(w_n))}\right|\le\text{\rm const}\cdot\ep nN \label{ineq5}
\end{align}
and
\begin{align}
\left|\vol{Q_{\omega_{nN}}}-\vol{Q(\tau_\omega(x_n),\tau_\omega(w_n))}\right|\le\text{\rm const}\cdot\ep nN. \label{ineq6}
\end{align}
Notice that inequalities (5) and (6) hold also in the presence of bad intervals as we only used property (5). We will use them in the general case as well.
 
Putting together the previous estimates (1)-(5) we get
\begin{align*}
\left|\vol{Q(\tau_\omega(x_n),\tau_\omega(w_n))}-\sum_{j\in[n]}{\vol{Q_{(\sigma^{jN}\omega)_N}}}\right|\le\text{\rm const}\cdot\ep nN
\end{align*}
Together with (6) this settles the basic case.

{\bf General case}. We now allow the presence of bad intervals. First, let us observe that the argument of the basic case, being independent of the past, immediately implies that if $I=[i,t]\subset[n]$ is an interval consisting entirely of good indices then we can find along $\tau_{\sigma^{iN}(\omega)}$ a pair of points $\ep|I|N<x<2\ep|I|N$ and $(1-2\ep)|I|N<w<(1-\ep)|I|N$ such that 
\begin{align}
\left|\vol{Q(\tau_{\sigma^{iN}(\omega)}(x),\tau_{\sigma^{iN}(\omega)}(w))}-\sum_{j\in I}{\vol{Q_{(\sigma^{jN}\omega)_N}}}\right|\le\text{\rm const}\cdot\ep |I|N. \label{ineq8}
\end{align}
Inequality (7) means, in words, that we can represent the ergodic sum over a good interval by a quasi-fuchsian manifold whose geodesic lies on the tracking ray of the interval. The idea of the general case is to proceed as in the basic case but with different building blocks.

The presence of bad intervals brings in some issues, whose nature is related to the way the the random walk deviates from the tracking ray, that we have to address. However, no new ingredients are needed, only a more careful choice of the interval subdivision. 

The problem can be summarized as follows: Consider a good interval $I_j$ and the adjacent bad interval $J_j$. Look at the deviation from the tracking ray of $I_j$ introduced by $J_j$. It might happen that the quasi-fuchsian manifold associated to the good interval $I_{j+1}$ is too small compared to the deviation and we are uncertain whether or not to include it in the gluing family. In order to get around the issue, we wait until the first time when the fellow traveling between the tracking rays of $I_j$ and $I_{j+1}$ is restored, discard all the good small intervals in between and replace the quasi-fuchsian manifold associated to $I_j$. So we start by refining the interval subdivision. 

{\bf Refinement of the interval subdivision}. Denote by $i_j<t_j$ the initial and the terminal indices in the $j$-th good interval $I_j=[i_j,t_j]$. We proceed inductively. Start with $I_1=[i_1=0,t_1]$ and $J_1=[t_1+1,i_2-1]$. Consider $I_2=[i_2,t_2]$. We determine a new $i^{\text{\rm new}}_3$ by the following condition
\[
i^{\text{\rm new}}_3:=\min\ga{\text{\rm $i>t_2+\ep(|I_1|+|J_1|)$ and $i$ is good}}.
\]
This requirement restores, by property (3), the fellow traveling between $\omega_{i_1N}\tau_{\sigma^{i_1N}(\omega)}$ and $\omega_{i_2N}\tau_{\sigma^{i_2N}(\omega)}$. That is $\omega_{i_1N}\tau_{\sigma^{i_1N}(\omega)}[(1+\ep)(|I_1|+|J_1|)N,\infty)$ and $\omega_{i_2N}\tau_{\sigma^{i_2N}(\omega)}[\ep(|I_1|+|J_1|)N,\infty)$ are $\delta$-fellow travelers (property (3)). The index $i^{\text{\rm new}}_3$ lies in some good interval $I_{j_3}$. We make the following replacement
\begin{align*}
I_3\longrightarrow &I^{\text{\rm new}}_3:=[i^{\text{\rm new}}_3,t_{j_3}]\\
J_2\longrightarrow &J_2^{\text{\rm new}}:=[t_2+1,i^{\text{\rm new}}_3-1]\\
 &=J^{\text{\rm old}}_2\sqcup I_3\sqcup\dots\sqcup J_{j_3-1}\sqcup[i_{j_3},i^{\text{\rm new}}_3-1].
\end{align*}

By our choice, if $j_3>3$, then the sum of the lengths $|J^{\text{\rm old}}_2|+|I_3|+\dots+|I_{j_3-1}|$ and $i^{\text{\rm new}}_3-i_{j_3}$ are controlled by $\ep(|I_1|+|J_1|)$. The length of $|J_{j_3-1}|$ can be, instead, arbitrarily long. Furthermore $|I^{\text{\rm new}}_3|\le|I_{j_3}|$. Observe that, for the new $J_2$ we have $|J^{\text{\rm new}}_2|=i^{\text{\rm new}}_3-t_2\le\ep(|I_1|+|J_1|)+|J_{j_3-1}|$. We leave untouched all the intervals after $I_{j_3}$, but we shift back the remaining indices $j\rightarrow 3+j-j_3$ for all $j>j_3$. We repeat the process and get inductively the new set of indices 
\[
i^{\text{\rm new}}_r:=\min\ga{\text{\rm $i>t^{\text{\rm new}}_{r-1}+\ep(|I^{\text{\rm new}}_{r-2}|+|J^{\text{\rm new}}_{r-2}|)$ and $i$ is good}}
\]
and intervals
\begin{align*}
I_r\longrightarrow &I^{\text{\rm new}}_r:=[i^{\text{\rm new}}_r,t_{j_r}]\\
J_{r-1}\longrightarrow &J_{r-1}^{\text{\rm new}}:=[t_{r-1}+1,i^{\text{\rm new}}_r-1]
\end{align*}
that satisfy $|J^{\text{\rm new}}_r|\le\ep(|I^{\text{\rm new}}_{r-2}|+|J^{\text{\rm new}}_{r-2}|)+|J_{j_{r+1}-1}|$. We end up with a new subdivision $[n]=I^{\text{\rm new}}_1\sqcup J^{\text{\rm new}}_1\sqcup\dots\sqcup I^{\text{\rm new}}_{k'}\sqcup J^{\text{\rm new}}_{k'}$ that still has the property
\begin{align*}
\sum_{t\le k'}{|J^{\text{\rm new}}_t|}\le\sum_{t\le k'}{\ep(|I^{\text{\rm new}}_{t-2}|+|J^{\text{\rm new}}_{t-2}|)+|J^{\text{\rm old}}_{j_{t+1}-1}|}\le\ep\sum_{t\le k'}{|J^{\text{\rm new}}_{t-2}|}+\ep n+\zeta n.
\end{align*}
Hence $\sum_{t\le k'}{|J^{\text{\rm new}}_t|}\le(\ep n+\zeta n)/(1-\ep)\le 4\ep n$ if $\zeta<\ep<1/2$. In particular the volumes corresponding to the new bad indices still add up to a small amount. In fact, by Lemma \ref{integrable}, we have 
\begin{align}
\sum_{i\in\bigsqcup{J^{\text{\rm new}}_j}}{\vol{Q_{\sigma^{iN}(\omega)_N}}}\le(CN+C)\sum_{i<k'}{|J^{\text{\rm new}}_i|}<\text{\rm const}\cdot\ep nN. \label{ineq9}
\end{align}  
For the sake of simplicity, after the refinement, we return to the previous notation $i_j:=i^{\text{\rm new}}_j$, $t_j:=t_j^{\text{\rm new}}$ and $I_j:=I^{\text{\rm new}}_j$, $J_j:=J^{\text{\rm new}}_j$, but assume the new properties.

{\bf Family III}. The proof can now proceed parallel to the basic case, so we only sketch the arguments. We define a family of quasi-fuchsian manifolds, one for every pair of adjacent intervals $I_j\sqcup J_j$, that can be glued to form a model for $T_{\omega_{nN}}$ that nearly computes its volume.

Proceed inductively. Start with $I_1\sqcup J_1=[0,t_1=|I_1|-1]\sqcup[t_1+1,i_2-1=|I_1|+|J_1|]$. Since $\tau_\omega$ is a good ray, we can find segments $[a_1,b_1]\subset[\ep |I_1|N, 2\ep |I_1|N]$ and $[c_1,d_1]\subset[(1+\ep)(|I_1|+|J_1|)N, (1+2\ep)(|I_1|+|J_1|)N]$ which are $\eta$-thick and have length at least $h$. Now consider $I_j\sqcup J_j$ for $j>1$. As in the basic case, we single out a pair of segments $[a_j,b_j]$, $[c_j,d_j]$ on the tracking ray of $\sigma^{i_jN}(\omega)$ normalized so that it starts at $O_{i_j}$. The first one, $[a_j,b_j]$, is determined by the condition that it is a $\delta$-fellow traveler of $[c_{j-1},d_{j-1}]$ contained in $[\ep(|I_j|+|J_j|)N, 2\ep(|I_j|+|J_j|)N]$ (see Figure \ref{fig:figure4} A). Here we are using in an essential way the properties of the refined interval and property (3) of good rays. The second one, $[c_j,d_j]$, is a $\eta$-thick $h$-long subsegment of $[(1+\ep)(|I_j|+|J_j|)N, (1+2\ep)(|I_j|+|J_j|)N]$. We simplify the notation by introducing
\begin{align*}
&A_j=\omega_{i_jN}\tau_{\sigma^{i_jN}(\omega)}(a_j), &B_j=\omega_{i_jN}\tau_{\sigma^{i_jN}(\omega)}(b_j),\\
&C_j=\omega_{i_jN}\tau_{\sigma^{i_jN}(\omega)}(c_j), &D_i=\omega_{i_jN}\tau_{\sigma^{i_jN}(\omega)}(d_j).
\end{align*}
We associate to $I_j\sqcup J_j$ the manifold $Q(A_j,D_j)$. 

The analogue of Lemma \ref{closing} holds word by word if we replace the old segments with the new ones, that is $[A_i,D_i]$ uniformly fellow travels $[O,O_n]$. 

By property (2), the latter uniformly fellow travels $l_n$, the axis of $\omega_{nN}$, along $\tau_\omega[\ep nN,(1-\ep)nN]$. In particular we can find $0<r<s<n$ such that $[A_r,D_r]$ and $[A_s,D_s]$ are, respectively, the first and the last segments that fellow travel $\tau_\omega[\ep nN,(1-\ep)nN]$ along some subsegments, which is terminal for the first and initial for the second. 

Up to discarding an initial (resp.  terminal) segment of $[A_r,D_r]$ (resp. $[A_s,D_s]$) of length smaller than $\ep |A_rD_r|$ (resp. $\ep |A_sD_s|$) we can assume that $[A_r,D_r]$ (resp. $[A_s,D_s]$) uniformly fellow travels subsegments of $\tau_\omega[\ep nN,(1-\ep)nN]$ and $l_n$ (as in Figure \ref{fig:figure4} B). The volumes of the associated quasi-fuchsian manifolds change at most by $\text{\rm const}\cdot\ep nN$ according to Proposition \ref{replacement}. 

We can also assume, by recurrence, that $[A_r,D_r]$ (resp. $[A_s,D_s]$) contains an initial (resp. terminal) $\eta$-thick subsegment $[A_r,B_r]$ (resp. $[C_s,D_s]$) of length at least $h$. We add the quasi-fuchsian manifold $Q(C_s,\omega_{nN}B_r)$ to the family. As in the basic case we have
\begin{align}
\vol{Q(C_s,\omega_{nN}B_r)}\le\text{\rm const}\cdot\ep nN. \label{ineq11}
\end{align}
Applying Proposition \ref{gluing} to the family $\{Q(A_j,D_j)\}_{j\in[r,s]}$ $\sqcup$ $\{Q(C_s,\omega_{nN}B_r)\}$ we can perform the cut and glue construction and get a manifold diffeomorphic to $T_{\omega_{nN}}$ with volume
\begin{align}
\left|\vol{T_{\omega_{nN}}}-\sum_{i\in[r,s]}{\vol{Q(A_i,D_i)}}-\vol{Q(C_s,\omega_{nN}B_r)}\right|\label{ineq12}\\
\le nV_0+\text{\rm const}\cdot\ep nN. \nonumber
\end{align}

The fellow traveling property of $\bigsqcup_{i<r}[A_i,D_i]$ (resp. $\bigsqcup_{i>s}[A_i,D_i]$) with $\tau_\omega[0,2\ep nN]$ (resp. $[\tau_\omega((1-\ep)nN),O_n]$) implies that $\sum_{i\not\in[r,s]}{d_\T(A_i,D_i)}\le 2\ep nN$ and, by Lemma \ref{replacement},
\begin{align}
\sum_{i\not\in[r,s]}{\vol{Q(A_i,D_i)}}\le\text{\rm const}\cdot\ep nN. \label{ineq10}
\end{align}

We compare now the volume of $Q(A_i,D_i)$ with the ergodic sum over the good interval $I_i$. Since the interval $I_j$ is good, we find on $\tau_{\sigma^{i_jN}(\omega)}$ two points $\ep|I_j|N<x_j<2\ep|I_j|N$ and $(1-2\ep)|I_j|N<w_j<(1-\ep)|I_j|N$ such that inequality (7) holds for $I=I_j$. Before going on, let us relax the notation, by introducing $X_j=\omega_{i_jN}\tau_{\sigma^{i_jN}(\omega)}(x_j)$ and $W_j=\omega_{i_jN}\tau_{\sigma^{i_jN}(\omega)}(w_j)$. We have
\begin{align}
\left|\vol{Q(X_j,W_j)}-\sum_{i\in I_j}{\vol{Q_{\sigma^{iN}(\omega)_N}}}\right|\le\text{\rm const}\cdot\ep|I_j|. \label{ineq13}
\end{align}

By Proposition \ref{replacement}, we have
\begin{align*}
\left|\vol{Q(A_j,D_j)}-\vol{Q(X_j,W_j)}\right|\le\kappa(d_\T(A_j,X_j)+d_\T(D_j,W_j))+\kappa.
\end{align*}
As $a_j,x_j\in[0,\ep(|I_j|+|J_j|)N]$ and $d_j,w_j\in[(1-\ep)|I_j|N,(1+2\ep)(|I_j|+|J_j|)N]$ we can continue the chain of inequalities with
\begin{align*}
\le\text{\rm const}\cdot\ep|I_j|N+\text{\rm const}\cdot|J_j|N.
\end{align*}
Adding all the contributions we get
\begin{align}
\left|\sum_{j\le k}{\vol{Q(A_j,D_j)}}-\sum_{j\le k}{\vol{Q(X_j,W_j)}}\right| \label{ineq14}\\
\le N\sum_{j\le k}{\text{\rm const}\cdot\ep|I_j|+\text{\rm const}\cdot|J_j|}\le\text{\rm const}\cdot\ep nN+\text{\rm const}\cdot\zeta nN.\nonumber
\end{align}
Putting together inequalities (10)-(13) and (5), (6) concludes the proof.
\end{proof}   

Theorem \ref{qf tracking} is now reduced to an application of the ergodic theorem which says that for $\mb{P}$-almost every $\omega$ the following limit exists finite 
\[
\lim_{n\rightarrow\infty}\frac{1}{nN}\sum_{j<n}{\text{vol}\pa{Q_{(\sigma^{jN}\omega)_N}}}=v_N.
\]
If $N$ and $\Omega_{\ep,\zeta,N}$ are as in Proposition \ref{approximation} then
\[
\limsup_{j\rightarrow\infty}{\frac{\vol{Q_{\omega_{jN}}}}{jN}}-\liminf_{j\rightarrow\infty}{\frac{\vol{Q_{\omega_{jN}}}}{jN}}\le\ep
\]
on $\Omega_{\ep,\zeta,N}$ which has measure at least $1-\zeta$. Applying Lemma \ref{real numbers} we get
\[
\limsup_{n\rightarrow\infty}{\frac{\vol{Q_{\omega_n}}}{n}}-\liminf_{n\rightarrow\infty}{\frac{\vol{Q_{\omega_n}}}{n}}\le\ep.
\]
This concludes the proof of Theorem \ref{qf tracking}.

\section{Some questions}
\label{questions}

We conclude with four questions.

\begin{qn}
What about other geometric invariants (e.g. diameter, systole, Laplace spectrum)? That is, given a geometric invariant $G(\bullet)$ of hyperbolic 3-manifolds, is there a function $f_G:\mb{N}\rightarrow\mb{R}$ such that $G(X_{\omega_n})/f_G(n)$ approaches a positive constant for almost every $\omega$? More specifically:
\begin{itemize}
\item{Does $\frac{1}{n}\cdot\text{\rm diam}(X_{\omega_n})$ converge?}
\item{Does $\log(n)^2\cdot\text{\rm systole}(T_{\omega_n})$ converge (see also \cite{ST})?}
\item{Does $n^2\cdot\lambda_1(X_{\omega_n})$ converge (see also \cite{BGH16}, \cite{HV})?}
\end{itemize}
\end{qn}

The strategy pursued in this article can be applied to the study of the asymptotic for other geometric invariants. The control one needs consists essentially of two parts: 
\begin{enumerate}[(i)]
\item{A comparison theorem for the geometric invariant computed for the negatively curved models and the underlying hyperbolic metric.}
\item{An understanding of the behaviour of the function that computes the geometric invariant for quasi-fuchsian manifolds.}
\end{enumerate}

In the next question we consider a different notion of randomness: Observe that, up to conjugacy, there is only a finite number of mapping classes with translation length at most $L$. Hence, for every fixed $L$, it makes sense to sample at random and uniformly a conjugacy class $\omega_L$ of a mapping class with translation length at most $L$. 

\begin{qn}
Does $\text{\rm vol}(T_{\omega_L})/L$ converge almost surely for $L\rightarrow\infty$? 
\end{qn}

A companion question for quasi-fuchsian manifolds is the following:

\begin{qn}
For which Teichmüller rays $\tau:[0,\infty)\rightarrow\T$ does the mean value $\text{\rm vol}(Q(\tau(0),\tau(t)))/t$ converge for $t\rightarrow\infty$?
\end{qn}

For pseudo-Anosov axes $l_\phi$ the limit exists and is equal to $\vol{T_\phi}/L(\phi)$ \cite{KoMc}, \cite{BB}. Theorem \ref{qf tracking} implies that it exists for every point and almost every ray with respect to exit measures of random walks. What about the Lebesgue measure on $\mc{PML}$ which is singular with respect to the exit measures \cite{GMT12}? 

The last question concerns the relation between hyperbolic volume and Teichmüller data: We know that $\text{\rm vol}(T_f)/d_{\text{\rm WP}}(f)\in[1/k(g),k(g)]$ (see \cite{Br03}, \cite{Br}, \cite{S13}). If we consider random walks, both numerator and denominator have a linear asymptotic $\text{\rm vol}(T_{\omega_n})/n\rightarrow v>0$ and $d_{\text{\rm WP}}(\omega_n)/n\rightarrow d>0$.

\begin{qn}
How does $v/d$ distribute? Does the ratio $v/d$ display an extremal behaviour?  
\end{qn}

One can ask the same for the Teichmüller translation lengths.

\medskip
\noindent
Mathematisches Institut der Universit\"at Bonn\\
Endenicher Allee 60, 53115 Bonn\\
Germany

\medskip
\noindent
e-mail: gviaggi@math.uni-bonn.de


\begin{thebibliography}{99}
\bibliographystyle{plain}

\bibitem{BGH16}
H. Baik, I. Gekhtman, U. Hamenst\"adt, {\em The smallest positive eigenvalue of fibered hyperbolic 3-manifolds}, arXiv:1608.07609 to appear in Proc. Lond. Math. Soc.

\bibitem{BP92}
R. Benedetti, C. Petronio, {\em Lectures on hyperbolic geometry}, Universitext, Springer Verlag, Berlin 1992.

\bibitem{Bers60}
L. Bers, {\em Simultaneous uniformization}, Bul. Amer. Math. Soc. {\bf 66}(1960), 94-97.

\bibitem{BCG98}
G. Besson, G. Courtois, S. Gallot, {\em Lemme de Schwarz r\'eel et applications g\'eom\'etriques}, Acta Math. {\bf 183}(1999), 145--169.

\bibitem{Br03}
J. Brock, {\em The Weil-Petersson metric and volumes of 3-dimensional hyperbolic convex cores}, J. Amer. Math. Soc. {\bf 16}(2003), 495-535.

\bibitem{Br}
J. Brock, {\em Weil-Petersson distance and volumes of mapping tori}, Comm. Anal. Geom. {\bf 11}(2003), 987-999.

\bibitem{BB}
J. Brock, K. Bromberg, {\em Inflexibility, Weil-Petersson distance, and volumes of fibered 3-manifolds}, Math. Res. Lett. {\bf 23}(2016), 649-674.

\bibitem{BCM} 
J. Brock, R. Canary, Y. Minsky, {\em The classification of Kleinian surface groups, II: The Ending Lamination Conjecture},  Ann. Math. {\bf 176}(2012), 1-149.

\bibitem{BMNS16} 
J. Brock, Y. Minsky, H. Namazi, J. Souto, {\em Bounded combinatorics and uniform models for hyperbolic 3-manifolds}, J. Topology {\bf 9}(2016), 451-501.  


\bibitem{CEG}
R. Canary, D. Epstein, B. Green, {\em Notes on notes by Thurston}, London Math. Soc. Lecture Notes 138, Cambridge University Press, Cambridge 2006.

\bibitem{DH18}
F. Dahmani, C. Horbez, {\em Spectral theorems for random walks on mapping class groups and {\rm Out($F_n$)}}, Int. Math. Res. Notices {\bf 9}(2018), 2693-2744.

\bibitem{DT06}
N. Dunfield, W. Thurston, {\em Finite covers of random 3-manifolds}, Invent. Math. {\bf 166}(2006), 457-521.

\bibitem{primer}
B. Farb, D. Margalit, {\em A primer on mapping class groups}, Princeton University Press (2012).

\bibitem{GMT12}
V. Gadre, J. Maher, G. Tiozzo, {\em Word length statistics for Teichmüller geodesics and singularity of harmonic measure}, Comment. Math. Helv. {\bf 92}(2012), 1-36.

\bibitem{H10}
U. Hamenst\"adt, {\em Stability of quasi-geodesics in Teichmüller space}, Geom. Dedicata {\bf 146}(2010), 101-116.

\bibitem{HV}
U. Hamenst\"adt, G. Viaggi, {\em Small eigenvalues of random 3-manifolds}, arXiv:1903.08031.

\bibitem{He01}
J. Hempel, {\em 3-manifolds as viewed from the curve complex}, J. Topology {\bf 40}(2001), 631-657.

\bibitem{KM}
V. Kaimanovich, H. Masur, {\em The Poisson boundary of the mapping class group}, Invent. Math. {\bf 125}(1996), 221--264.

\bibitem{KoMc}
S. Kojima, G. McShane, \emph{Normalized entropy versus volume for pseudo-Anosov}, Geom. Topol. {\bf 22}(2018), 2403-2426.


\bibitem{L74}
M. Linch, \emph{A comparison of metrics on Teichmüller space}, Proc. Amer. Math. Soc. {\bf 43}(1974), 349-352.

\bibitem{Ma}
J. Maher, {\em Random walks on the mapping class group}, Duke Math. J. {\bf 156}(2011), 429-468.

\bibitem{Ma10}
J. Maher, {\em Random Heegaard splittings}, J. Topology {\bf 3}(2010), 997-1025.

\bibitem{MS18}
J. Maher, S. Schleimer, {\em The compression body graph has infinite diameter}, arXiv:1803.06065.

\bibitem{MT14}
J. Maher, G. Tiozzo, {\em Random walks on weakly hyperbolic groups}, J. reine angew. Math. {\bf 742}(2018), 187-239.

\bibitem{MM99}
H. Masur, Y. Minsky, {\em Geometry of the complex of curves I: Hyperbolicity}, Invent. Math. {\bf 138}(1999), 103-149.

\bibitem{MM00}
H. Masur, Y. Minsky, {\em Geometry of the complex of curves II: Hierarchical structure}, Geom. Funct. Anal. {\bf 10}(2000), 902-974.

\bibitem{M96} 
Y. Minsky, {\em Quasi-projections in Teichmüller space}, J. reine angew. Math. {\bf 473}(1996),
121-136.

\bibitem{M10} 
Y. Minsky, {\em The classification of Kleinian surface groups, I: Models and bounds}, Ann. Math. {\bf 171}(2010), 1-107.

\bibitem{Na05} 
H. Namazi, {\em Heegaard splittings and hyperbolic geometry}, PhD Thesis, Stony Brook University.

\bibitem{NS09} 
H. Namazi, J. Souto, {\em Heegaard splittings and pseudo-Anosov maps}, Geom. Funct. Anal. {\bf 19}(2009), 1195-1228.

\bibitem{R14}
K. Rafi, {\em Hyperbolicity in Teichmüller space}, Geom. Topol. {\bf 17}(2014), 3025-3053.

\bibitem{Riv}
I. Rivin, {\em Statistics of random 3-manifolds occasionally fibering over the circle}, arXiv:1401.5736v4.

\bibitem{S13}
J-M. Schlenker, {\em The renormalized volume and the volume of the convex core of quasi-fuchsian manifolds}, Math. Res. Lett. {\bf 20}(2013), 773-786.

\bibitem{ST}
A. Sisto, S. Taylor, {\em Largest projections for random walks and shortest curves for random mapping tori}, arXiv:1611.07545 to appear in Math. Res. Lett.

\bibitem{Th}
W. Thurston, {\em Hyperbolic structures on 3-manifolds, II: Surface groups and 3-manifolds which fiber over the circle}, arXiv:math/9801045.

\bibitem{T15}
G. Tiozzo, {\em Sublinear deviation between geodesics and sample paths}, Duke Math. J. {\bf 164}(2015), 511-539.

\end{thebibliography}
\end{document}